\documentclass[11pt]{article}
\usepackage{amsfonts}
\usepackage{mathrsfs}
\usepackage{amsthm}
\usepackage{amsmath}
\usepackage{amssymb}
\usepackage{latexsym}
\usepackage{graphicx}
\usepackage{color,varioref}
\usepackage{longtable}
\usepackage{cite}
\usepackage{enumerate}
\usepackage{appendix}
\definecolor{red}{rgb}{0.9,0,0}

\definecolor{green}{rgb}{0,0.9,0}

\definecolor{blue}{rgb}{0,0,0.9}

\def\grad{\nabla}

\def\grad{\nabla}

\def\gph{\mbox{\rm gph}\,}

\def\dom{\mbox{\rm dom}\,}

\def\ox{\overline{x}}

\def\tto{\rightrightarrows}

\def\Bar{\overline}

\def\ve{\varepsilon}
\def\epsilon{\varepsilon}
\def\ox{\bar{x}}

\def\emp{\emptyset}

\def\Limsup{\mathop{{\rm Lim}\,{\rm sup}}}

\newtheorem{theorem}{Theorem}[section]
\newtheorem{proposition}{Proposition}[section]
\newtheorem{lemma}{Lemma}[section]

\newtheorem{definition}{Definition}[section]
\newtheorem{assumption}{Assumption}[section]

\newtheorem{example}{Example}[section]

\begin{document}

\title{\bf Characterizations of Strong Variational Sufficiency\\ in General Models of Composite Optimization}

\author{Boris S. Mordukhovich\footnote{Department of Mathematics, Wayne State University, Detroit, MI 48202 (aa1086@wayne.edu). Research of this author was partly supported by the US National Science Foundation under grant DMS-2204519, by the Australian Research Council under Discovery Project DP-190100555, and by Project 111 of China under grant D21024.} \quad Peipei Tang\footnote{School of Computer and Computing Science, Hangzhou City University, Hangzhou, China (tangpp@hzcu.edu.cn).}\quad Chengjing Wang\footnote{School of Mathematics, Southwest Jiaotong University, Chengdu, China (renascencewang@hotmail.com).}}

\maketitle

\begin{abstract}

This paper investigates a recently introduced notion of strong variational sufficiency in optimization problems whose importance has been highly recognized in optimization theory, numerical methods, and applications. We address a general class of composite optimization problems and establish complete characterizations of strong variational sufficiency for their local minimizers in terms of a generalized version of the strong second-order sufficient condition (SSOSC) and the positive-definiteness of an appropriate generalized Hessian of the augmented Lagrangian calculated at the point in question. The generalized SSOSC is expressed via a novel second-order variational function, which reflects specific features of nonconvex composite models. The imposed assumptions describe the spectrum of composite optimization problems covered by our approach while being constructively implemented for nonpolyhedral problems that involve the nuclear norm function  and the indicator function of the positive-semidefinite cone without any constraint qualifications. 
\end{abstract}
\vspace*{0.1in}

\section{Introduction}\label{intro}

The notion of {\em variational} and {\em strong variational sufficiency} for local minimizers of general optimization problems written in the unconstrained extended-real-valued format  was introduced by Rockafellar \cite{Rockafellar2019a}. It is closely related to Rockafellar's notion of (strong) {\em variational convexity} of lower semicontinuous (l.s.c.) functions \cite{Rockafellar2019} meaning that the subdifferential graph of the function in question locally behaves as the subgradient mapping associated with some (strongly) convex l.s.c.\ function. Both notions of variational sufficiency and variational convexity, as well as their strong counterparts, reflect ``primal-dual" perspectives of local optimality. These properties enhance our understanding of optimization models by demonstrating how their objective functions can exhibit convex-type features even in more complex settings. Moreover, they provide a robust theoretical foundation for the development of numerical algorithms to efficiently solve nonsmooth and nonconvex optimization problems.

In this paper, we mainly concentrate on the study of strong variational sufficiency for local minimizers of the following problem of (constrained) {\em composite optimization}
\begin{align*}
(COP)\qquad\min_{x\in\mathcal{R}^{n}} f_{0}(x)+g(F(x)) 
\end{align*}
under the {\em standing assumptions}: $f_{0}:\mathcal{R}^{n}\rightarrow\mathcal{R}$ and $F:\mathcal{R}^{n}\rightarrow\mathcal{R}^{m}$ are ${\cal C}^2$-smooth mappings, and  $g:\mathcal{R}^{m}\rightarrow \overline{\mathcal{R}}:=\mathcal{R}\cup\{\infty\}$ is an l.s.c.\  convex function. The class of optimization problems in form (COP) is very broad including linear and nonlinear programs, second-order cone and positive-semidefinite programs, etc. In the {\em nondegenerate} case where the Jacobian matrix of $F$ is of full rank at the reference point, the 
second-order characterization of strong variational sufficiency for local optimality in COPs is obtained in \cite{Mordukhovich2023}. In the (polyhedral) case of nonlinear programs (NLPs), the aforementioned nondegeneracy condition reduces to the classical linear independence constraint qualification (LICQ), and the characterization of \cite{Mordukhovich2023} is explicitly formulated via the classical {\em strong second-order sufficient condition} (SSOSC). We also mention the abstract {\em quadratic bundle} characterizations in \cite{Rockafellar2023} formulated in the SSOSC form  as well in the case of NLPs. For the class of (nonpolyhedral) nonlinear semidefinite programs (SDPs), characterizations of strong variational sufficiency were developed in \cite{DingChao2023} via the corresponding version of SSOSC. 
Moreover, the obtained results for strong variational sufficiency were used for establishing local convergence rates for numerical algorithms of the proximal, augmented Lagrangian, and Newtonian types; see \cite{Khanh2025,Rockafellar2019a,Rockafellar2022,DingChao2023}. Major {\em open questions} remain about explicit characterizations of strong variational sufficiency for general nonpolyhedral composite optimization problems in terms of their given data, and particularly about the equivalence between strong variational sufficiency for local optimality in (COP) and appropriate versions of SSOSC. 

This paper develops an efficient approach to resolve such issues based on advanced machinery of variational analysis and second-order generalized differentiation. Our main goal is to reveal general assumptions on the initial data of (COP) that ensure the desired equivalence of strong variational sufficiency for local optimality to the newly formulated SSOSC and the related positive-definiteness of a certain generalized Hessian of the augmented Lagrangian function associated with (COP) without any constraint qualification. Furthermore,  we specifically illustrate the fulfillment of the most involved assumption in our analysis with constructive calculations in the two major nonpolyhedral cases: when $g$ is either the indicator function of the positive-semidefinite cone, or the nuclear norm function. The obtained calculations lead us to explicit characterizations of strong variational sufficiency in the corresponding COPs, where the crucial case of nuclear norm-based problems has never been considered in the literature.
 
The rest of the paper is organized as follows. Section~\ref{sec:prel} presents the needed preliminaries from variational analysis and generalized differentiation together with the formulation and discussion of the major assumptions on the data of (COP) to provide desired characterizations. In Section~\ref{sec:prox}, we evaluate the limiting coderivative of the proximal mapping that is broadly used in what follows. Section~\ref{sec:ssosc} introduces a novel second-order variational function, which is used in the formulation and investigation of the generalized SSOSC for the composite problem (COP). Section~\ref{sec:char} contains the main results on characterizing strong variational sufficiency for (COP).  The concluding remarks and some topics of the future research are discussed in Section~\ref{sec:conc}. Finally, Appendix~\ref{appendix-A} presents explicit calculations of all the ingredients involving into the major assumptions and characterizations for the SDP cone and nuclear norm functions in (COP).\vspace*{0.05in}

Throughout the  paper, we use the standard notation of finite-dimensional variational analysis and generalized differentiation; see, e.g., \cite{Mordukhovich2018,Rockafellar1998}. Recall that the closed unit ball in $\mathcal{R}^{m}$ is denoted by $\mathbb{B}_{\mathcal{R}^{m}}$, while $\mathbb{B}(x,\varepsilon):=\{u\in\mathcal{R}^{m}\, |\, \|u-x\|\leq\varepsilon\}$ for some $x\in\mathcal{R}^{m}$ and $\varepsilon>0$. For a set $C\subset\mathcal{R}^{m}$, the indicator function $\delta_{C}$ of $C$ is defined by $\delta_{C}(x):=0$ if $x\in C$ and $\delta_{C}(x):=\infty$. Given a point $x\in\mathcal{R}^{m}$, we use 
$\operatorname{\Pi}_{C}(x)$ to denote the projection of $x$ onto $C$.
For a sequence $\{x^{k}\}$, the notation $x^{k}{\stackrel{C}{\rightarrow}}x$ means that $x^{k}\rightarrow x$ with $x^{k}\in C$, while  
$x^{k}{\stackrel{N}{\rightarrow}}x$ indicates that $x^{k}\rightarrow x$ with $k\in N\subset \mathbb N:=\{1,2,\ldots\}$. The smallest cone containing $C$ is $\operatorname{pos}C=\{0\}\cup\{\lambda x\, |\, x\in C,\, \lambda>0\}$ with $\operatorname{pos}C:=\{0\}$ if $C=\emp$. The intersection of all the convex sets containing $C$ is called the convex hull of $C$ and is denoted by $\operatorname{conv}C$. For a nonempty closed convex cone $C$, the polar cone of $C$ is $C^{\circ}:=\{y\,|\, \langle x,y\rangle\leq 0\;\mbox{ for all }\;x\in C\}$, and the gauge of $C$ is $\gamma_{C}(x):=\inf\{\lambda\ge 0\,|\,x\in\lambda C\}$ provided that $0\in C$. The notation $\operatorname{Diag}(x)$ signifies the diagonal matrix with vector $x$ on the diagonal. Denote the $n\times n$ identity matrix as $I_{n}$, where the subscript is omitted if the dimension is evident from the context. For a given matrix $A\in\mathcal{R}^{m\times n}$, let $\operatorname{ker}(A):=\{x\in\mathcal{R}^{n}\,|\,Ax=0\}$, and let $A^{\dagger}$ stand for the Moore-Penrose inverse of $A$. Given index sets $\alpha\subset\{1,\ldots,m\}$ and $\beta\subset\{1,\ldots,n\}$, denote by $A_{\alpha\beta}$ the $|\alpha|\times|\beta|$ submatrix of $A$ obtained by removing all the rows of $A$ not in $\alpha$ and all the columns of $A$ not in $\beta$. For a set-valued mapping/multifunction $S\colon\mathcal{R}^{n}\tto\mathcal{R}^{m}$, the (Painlev\'e-Kuratowski)  outer limit of $S$ at $\ox$ is 
\begin{equation}\label{pk}
\Limsup_{x\to\ox}S(x):=
\big\{y\in\mathcal{R}^{m}\;\big|\;\exists\,x_k\to\ox,\;\exists\,y_k\to y\;\mbox{ with }\;y_k\in S(x_k)\big\},
\end{equation}
while the standard upper limit of scalars are denoted by `$\limsup$'. Finally, we use `$\subset$' to indicate that the set on the left-hand side is smaller than or equal to the one on the right.

\section{Major Preliminaries and Assumptions}\label{sec:prel}

In this section, we review fundamental concepts of variational analysis and generalized differentiation, which are extensively employed in the subsequent sections. For a comprehensive treatment of these topics, we refer the reader to the books \cite{Bonnans2000,Mordukhovich2006,Mordukhovich2024,Mordukhovich2018,Rockafellar1998}. 

Considering an extended real-valued function $r:\mathcal{R}^{m}\rightarrow\overline{\mathcal{R}}$, the effective {\em domain} and the {\em epigraph} of this function are defined by
\begin{align*}
\operatorname{dom}r:=\{x\in\mathcal{R}^{m}\,|\,r(x)<\infty\}\quad\mbox{and}\quad\operatorname{epi}r:=\{(x,y)\in\mathcal{R}^{m}\times\mathcal{R}\,|\,x\in\operatorname{dom}r,\,r(x)\leq y\},
\end{align*}
respectively.
The (Fenchel) {\em conjugate function} of $r$ at $x\in\mathcal{R}^{m}$ is given by
\begin{align*}
r^{*}(x):=\sup_{u\in\operatorname{dom}r}\left\{\langle x,u\rangle-r(u)\right\}.
\end{align*}
The function $r$ is called {\em proper} if $\operatorname{dom}r\neq\emptyset$. For a proper  l.s.c.\ function $r$ and a constant $\sigma>0$, the {\em Moreau envelope} $e_{\sigma r}$ and the {\em proximal mapping} $\operatorname{Prox}_{\sigma r}$ corresponding to $r$ and $\sigma$ are 
\begin{align}
e_{\sigma r}(x)&:=\inf_{u\in\mathcal{R}^{m}}\left\{r(u)+\frac{1}{2\sigma}\|u-x\|^{2}\right\},\label{moreau}\\
\operatorname{Prox}_{\sigma r}\label{prox}(x)&:=\mathop{\operatorname{\arg\min}}\limits_{u\in\mathcal{R}^{m}}\left\{r(u)+\frac{1}{2\sigma}\|u-x\|^{2}\right\}.
\end{align}
When $r$ is proper, l.s.c., and convex, we know from \cite[Theorem~2.26]{Rockafellar1998} that \eqref{moreau} is convex and continuously differentiable (${\cal C}^1$-smooth) with the gradient
\begin{align*}
\nabla e_{\sigma r}(x)=\frac{1}{\sigma}(x-\operatorname{Prox}_{\sigma r}(x)).
\end{align*}
Moreover, it follows in this case from \cite[Theorem~31.5]{Rockafellar1970} that \eqref{prox} is single-valued, Lipschitz continuous and satisfies the Moreau identity
\begin{align*}
\operatorname{Prox}_{\sigma r}(x)+\sigma\operatorname{Prox}_{\sigma^{-1}r^{*}}(\sigma^{-1}x)=x\;\mbox{ whenever }x\in\mathcal{R}^{m}.
\end{align*}

Given $r:\mathcal{R}^{m}\rightarrow\overline{\mathcal{R}}$ and $x\in\dom r$, 
the (Fr\'echet) {\em regular subdifferential} of $r$ at $x$ is 
\begin{align*}
\widehat{\partial}r(x):=\left\{y\in\mathcal{R}^{m}\,\left|\,r(x')\geq r(x)+\langle y,x'-x\rangle+o(\|x'-x\|)\right.\right\},
\end{align*}
while the (Mordukhovich)  {\em limiting subdifferential} of $r$ at $x$ can be equivalently defined by
\begin{align*}
\partial r(x):=\left\{y\in\mathcal{R}^{m}\,\left|\,\exists\,x^{k}\rightarrow x\;\mbox{ with }\;r(x^{k})\rightarrow r(x),\;y^{k}\in\widehat{\partial}r(x^{k}),\;y^{k}\rightarrow y\right.\right\}.
\end{align*}
The {\em subderivative}  of $r$ at $x\in\dom r$ in the direction $d\in\mathcal{R}^{m}$ is defined by
\begin{align}\label{subder}
dr(x)(d):=\liminf_{\stackrel{t\downarrow 0}{d'\rightarrow d}}\frac{r(x+td')-r(x)}{t}.
\end{align}
The {\em parabolic subderivative} of $r$ at $x$ relative to a vector $d\in\mathcal{R}^{m}$ for which \eqref{subder} is finite and to any vector $w\in\mathcal{R}^{m}$ is given by
\begin{align}\label{parab-sub}
d^{2}r(x)(d,w):=\liminf_{\stackrel{t\downarrow 0}{w'\rightarrow w}}\frac{r(x+td+\frac{1}{2}t^{2}w')-r(x)-tdr(x)(d)}{\frac{1}{2}t^{2}}.
\end{align}
The function $r$ is said to be {\em parabolically epi-differentiable} at $x$
for $d$ if $\operatorname{dom}d^{2}r(x)(d,\cdot)\neq\emptyset$
and for every $w\in\mathcal{R}^{m}$ and every $t_{k}\downarrow 0$, there exists $w_{k}\rightarrow w$ such that 
\begin{align*}
d^{2}r(x)(d,w)=\lim_{k\rightarrow\infty}\frac{r(x+t_{k}d+\frac{1}{2}t_{k}^{2}w_{k})-r(x)-t_{k}dr(x)(d)}{\frac{1}{2}t_{k}^{2}}.
\end{align*}
Given $t>0$, $x\in\operatorname{dom}r$, and $u\in\partial r(x)$, consider
\begin{align*}
\Delta_{t}^{2}r(x,u)(d):=\frac{r(x+t d)-r(x)-t\langle u,d\rangle}{\frac{1}{2}t^{2}}
\end{align*}
and define the {\em second-order subderivative} of $r$ at $x$ for $u$ and $d\in\mathcal{R}^{m}$ by
\begin{align}\label{2subder}
d^{2}r(x,u)(d):=\liminf_{{t\downarrow 0}\atop{d'\rightarrow d}}\Delta_{t}^{2}r(x,u)(d').
\end{align}
which can be equivalently written via the lower epi-limit
\begin{align*}
d^{2}r(x,u)=\mathop{\operatorname{e-lim\inf}}_{t\downarrow 0}\Delta_{t}^{2}r(x,u).
\end{align*}
We say that $r$ is {\em twice epi-differentiable} at $x$ for $u$ if $\Delta_{t}^{2}r(x,u)$ epi-converges to $d^{2}r(x,u)$ as $t\downarrow 0$.

\begin{definition}\label{par-reg}
Let $r:\mathcal{R}^{m}\rightarrow \overline{\mathcal{R}}$ be a proper function with $x\in\operatorname{dom}r$ and $u\in\partial r(x)$. It is said that $r$ is {\sc parabolically regular} at $x$ in the direction $d$ for $u$ if the relationship
\begin{align*}
\inf_{w\in\mathcal{R}^{m}}\left\{d^{2}r(x)(d,w)-\langle u,w\rangle\right\}=d^{2}r(x,u)(d)
\end{align*} 
holds for the second-order constructions \eqref{parab-sub} and \eqref{2subder}.
\end{definition}
The concept of parabolic regularity was first formulated in  \cite[Definition~13.65]{Rockafellar1998} (cf.\ also \cite[Definition~3.102]{Bonnans2000}) and has been comprehensively studied more recently in \cite{Mohammadi2020,Mohammadi2021,Mohammadi2022}.\vspace*{0.03in}

Given further an open set $\Omega\subset\mathcal{R}^{n}$ and a Lipschitz continuous mapping $S:\Omega\rightarrow\mathcal{R}^{m}$, recall by the classical Rademacher’s theorem (see, e.g., \cite[Lemma~1.83]{Mordukhovich2024} with a simple proof) that $S$ is differentiable almost everywhere in $\Omega$ and denote by $\Omega_{S}\subset\Omega$ the set of points where $S$ is differentiable. For any $x\in\Omega_{S}$, let $S'(x)\in\mathcal{R}^{m\times n}$ be the Jacobian matrix of $S$ at $x$ with the adjoint/transpose matrix denoted by $S'(x)^T$. The (Clarke) {\em generalized Jacobian} $JS(x)$ of $S$ at $x$ is the convex hull of the {\em limiting Jacobian} set
of $S$ at $x$ defined by
\begin{align}\label{Bsub}
\Bar\grad S(x):=\left\{V\in\mathcal{R}^{m\times n}\,\left|\,\exists\,  x^{k}{\stackrel{\Omega_{S}}{\longrightarrow}} x,\;S'(x^{k})\rightarrow V\right.\right\}.
\end{align}

Next we consider the tangent and normal cones to sets used in what follows. Given a set $C\subset\mathcal{R}^m$ and a point $x\in C$, the (Bouligand-Severi) {\em tangent cone} to $C$ at $x$ is defined by
\begin{align}\label{tan}
\mathcal{T}_{C}(x):=\Limsup_{t\downarrow 0}\frac{C-x}{t}=\left\{w\in\mathcal{R}^{m}\,\left|\,\exists\, t^{k}\downarrow 0,  x^{k}{\stackrel{C}{\longrightarrow}} x, \frac{x^{k}-x}{t^{k}}\rightarrow w \right.\right\}
\end{align}
via the outer limit \eqref{pk}. The (Fr\'echet) {\em regular normal cone} to $C$ at $x$ is equivalently given by
\begin{align}\label{rnc}
\widehat{\mathcal{N}}_{C}(x):=\left\{y\in\mathcal{R}^{m}\, \left|\, \limsup_{{x'{\stackrel{C}{\longrightarrow}} x}\atop{x'\neq x}}\frac{\langle y,x'-x\rangle}{\|x'-x\|}\leq 0\right.\right\}=\mathcal{T}_{C}^{\circ}(x),
\end{align}
and the (Mordukhovich) {\em limiting normal cone} to $C$ at $x$ is
\begin{align}\label{lnc}
\mathcal{N}_{C}(x):=\Limsup_{x'{\stackrel{C}{\longrightarrow}} x}\widehat{\mathcal{N}}_{C}(x'),
\end{align}
which is equivalent to the original definition 
\begin{align*}
\mathcal{N}_{C}(x)=\Limsup_{x'\rightarrow x}\left\{\operatorname{pos}(x'-\operatorname{\Pi}_{C}(x'))\right\}
\end{align*}
from \cite{MORDUKHOVICH1976} if $C$ is locally closed around $x$. Note that, despite the nonconvexity of the sets in \eqref{lnc}, the limiting normal cone and the associated constructions for functions and multifunctions enjoy {\em full calculus} in contrast to the convex-valued  regular normals \eqref{rnc} and the like. 

For a multifunction $S:\mathcal{R}^{n}\rightrightarrows\mathcal{R}^{m}$ with the domain $\dom S:=\{x\in\mathcal{R}^n\,|\,S(x)\ne\emp\}$ and the graph $\gph S:=\{(x,y)\in\mathcal{R}^n\times\mathcal{R}^m\,|\,y\in S(x)\}$, the following generalized derivatives of $S$ are defined via the corresponding tangent and normal cones \eqref{tan}--\eqref{lnc} to the graph.

\begin{definition}\label{def:cod}
Given a multifunction $S:\mathcal{R}^{n}\rightrightarrows\mathcal{R}^{m}$ and $x\in\operatorname{dom}S$, the {\sc graphical derivative} of $S$ at $x$ for $u\in S(x)$ is the mapping $DS(x,u):\mathcal{R}^{n}\rightrightarrows\mathcal{R}^{m}$ defined by
\begin{align}\label{gder}
z\in DS(x,u)(w)\quad\Longleftrightarrow\quad (w,z)\in\mathcal{T}_{\operatorname{gph}S}(x,u).
\end{align}
The {\sc regular coderivative} of $S$ at $x$ for $u\in S(x)$ is  defined by
\begin{align*}
z\in \widehat{D}^{*}S(x,u)(w)\quad\Longleftrightarrow\quad (z,-w)\in\widehat{\mathcal{N}}_{\operatorname{gph}S}(x,u).
\end{align*}
The {\sc limiting coderivative} of $S$ at $x$ for $u\in S(x)$ is defined by
\begin{align}\label{lim-cod}
z\in D^{*}S(x,u)(w)\quad\Longleftrightarrow\quad (z,-w)\in\mathcal{N}_{\operatorname{gph}S}(x,u).
\end{align}
When $S$ is single-valued at $x$, we skip $u$ in the notation above.
\end{definition}

If $S\colon\mathcal{R}^n\to\mathcal{R}^m$ is Lipschitz continuous around $\ox$, then there is the well-known relationship \cite{Mordukhovich2006} between the limiting coderivative and generalized Jacobian of $S$ at $x$:
\begin{equation}\label{cod-cl}
{\rm conv}D^*S(x)(w)=\big\{A^Tw\,\big|\;A\in JS(x)\big\}\;\mbox{ for all }\;w\in\mathcal{R}^m.
\end{equation}

Let us further mention the following simple albeit useful relationships for an arbitrary set-valued mapping  $S:\mathcal{R}^{n}\rightrightarrows\mathcal{R}^{m}$ and its inverse $S^{-1}$:
\begin{align}\label{eq:inverse-graphical-derivative}
\begin{aligned}
z\in DS(x,u)(w)\,&\Longleftrightarrow\, w\in D(S^{-1})(u,x)(z),\\ 
z\in D^{*}S(x,u)(w)\,&\Longleftrightarrow\, -w\in D^{*}(S^{-1})(u,x)(-z)
\end{aligned}
\end{align}
together with the elementary sum rules for the summation of $S:\mathcal{R}^{n}\rightrightarrows\mathcal{R}^{m}$ and a smooth mapping $h:\mathcal{R}^{n}\rightarrow\mathcal{R}^{m}$ that are formulated as
\begin{align}\label{eq:graphical-sum}
\begin{aligned}
\mathcal{T}_{\operatorname{gph}(h+S)}(x,u)&=\big\{(w,h'(x)w+z)\,\big|\, (w,z)\in\mathcal{T}_{\operatorname{gph} S}(x,u-h(x))\big\},\\
\mathcal{N}_{\operatorname{gph}(h+S)}(x,u)&=\left\{(w- h'(x)^{T}z,z)\,|\,(w,z)\in\mathcal{N}_{\operatorname{gph}S}(x,u-h(x))\right\}.
\end{aligned}
\end{align}
By the definitions in \eqref{gder} and \eqref{lim-cod}, the latter representations readily imply that, respectively,
\begin{align}\label{eq:coderivative-sum}
\begin{aligned}
D(h+S)(x,u)(z)&=h'(x)z+DS(x,u-h(x))(z),\\
D^{*}(h+S)(x,u)(z)&=h'(x)^{T}z+D^{*}S(x,u-h(x))(z).
\end{aligned}
\end{align}

Next we revisit the concepts of (strong)  variational convexity of functions and (strong) variational sufficiency for local optimality introduced in \cite{Rockafellar2019,Rockafellar2019a}.

\begin{definition}\label{def:varconv}
Let $r:\mathcal{R}^{m}\rightarrow \overline{\mathcal{R}}$ be an l.s.c.\ function with $\bar{x}\in\operatorname{dom}r$ and $\bar{u}\in\partial r(\bar{x})$. We say that $r$ is {\sc variationally convex} at $\bar{x}$ for $\bar{u}$ if for some convex neighborhood $X\times Y$ of $(\bar{x},\bar{u})$, there exist an l.s.c.\ convex function $\hat{r}\leq r$ on $X$ and a constant $\ve>0$ such that 
\begin{align*}
[X_{\varepsilon}\times Y]\cap\operatorname{gph}\partial r=[X\times Y]\cap\operatorname{gph}\partial\hat{r}\ \mbox{and}\ r(x)=\hat{r}(x)\ \mbox{at the common elements}\ (x,u),
\end{align*}
where $X_{\varepsilon}:=\left\{x\in X\,\left|\,r(x)<r(\bar{x})+\varepsilon\right.\right\}$.
The function $r$ is called {\sc variationally strongly convex} at $\bar{x}$ for $\bar{u}$ with modulus $\sigma>0$ if $\hat{r}$ above is strongly convex on $X$ with that modulus.
\end{definition}

To define the notion of variational sufficiency and its strong counterpart for (COP), it is convenient to incorporate  a perturbation parameter therein and recast (COP) in the form:
\begin{align*}
(P)\qquad\min\ f_{0}(x)+g(F(x)+s)\quad\mbox{ s.t. }\;s=0.
\end{align*}

\begin{definition}\label{varsuf} We say that the {\sc variational sufficiency} for local optimality in $(P)$ holds at $\bar{x}$ relative to $\bar{u}$ satisfying the KKT system \eqref{eq:kkt-problem-P} if there exists $\sigma>0$ such that the function $G_{\sigma}$ is variationally convex with respect to the pair $((\bar{x},0),(0,\bar{u}))\in\operatorname{gph\,\partial G_{\sigma}}$, where
\begin{align}\label{G}
G_{\sigma}(x,s):=f_{0}(x)+g(F(x)+s)+\frac{\sigma}{2}\|s\|^{2}.
\end{align}
{\sc Strong variational sufficiency} for local optimality in $(P)$ holds at $\bar{x}$ relative to $\bar{u}$ with modulus $\kappa>0$ if $G_{\sigma}$ in \eqref{G} is variationally strongly convex at $\bar{x}$ for $\bar{u}$ with that modulus.
\end{definition}

The general standing assumptions on the data of (COP) were given in Section~\ref{intro}. However, they are not enough to guarantee the desired equivalence. Now we formulate and discuss some additional assumptions on the extended-real-valued function $g$ in (COP), which will be used in the subsequent sections to establish constructive second-order characterizations of variational sufficiency for local optimality in (COP). Although these assumptions will be mainly applied below to the function $g$ in (COP), we'll formulate them for an arbitrary extended-real-valued l.s.c.\ convex function to make it convenient to study some properties and interrelations important for their own sake, independently of applications to (COP). We start with the following one.

\begin{assumption}\label{assump-1}
Let $r:\mathcal{R}^{m}\rightarrow \overline{\mathcal{R}}$ be an l.s.c.\ convex function with  $\bar{x}\in\dom r$ and $\bar{u}\in\partial r(\bar{x})$. Suppose that this function is parabolically regular at $\bar{x}$ for $\bar{u}$ and parabolically epi-differentiable at $\bar{x}$ for all $w\in\{d\,|\,dr(\bar{x})(d)=\langle\bar{u},d\rangle\}$ with the subderivative $dr(\bar{x})$ taken from \eqref{subder}.
\end{assumption}

It has been well recognized in variational analysis that Assumption~\ref{assump-1} is not restrictive and is satisfied for broad classes of extended-real-valued functions that overwhelmingly appear in composite optimization; see \cite{Mohammadi2021,Mohammadi2022,Mohammadi2020} for a collection of major results in this direction. In particular, Assumption~\ref{assump-1} holds if $r$ is ${\cal C}^2$-{\em cone reducible} in the sense of the definition below.

\begin{definition}
Given two closed convex sets $C\subset\mathcal{R}^{m}$ and $K\subset\mathcal{R}^{l}$ with $l\in\mathbb N\cup\{\infty\}$, it is said that $C$ is ${\cal C}^{l}$-reducible to $K$ at a point $x\in C$ if there exist a neighborhood $U$ of $x$ and an $l$-times continuously differentiable mapping $\Xi:U\rightarrow\mathcal{R}^{l}$ 
such that $\Xi'(x):\mathcal{R}^{m}\rightarrow\mathcal{R}^{l}$ is surjective and $C\cap U=\{x\in U\, |\, \Xi(x)\in K\}$. The reduction is pointed if the tangent cone $\mathcal{T}_{K}(\Xi(x))$ is pointed with $\Xi(x)=0$. If in addition $K-\Xi(x)$ is a pointed closed convex cone, then $C$ is ${\cal C}^{l}$-cone reducible at $x$. Finally, we say that a closed proper convex function $r:\mathcal{R}^{m}\rightarrow\overline{\mathcal R}$ is ${\cal C}^{2}$-{\sc cone reducible} at $x$ if the set $\operatorname{epi}r$ is ${\cal C}^{2}$-cone reducible at $(x,r(x))$. Moreover, it is said that $r$ is ${\cal C}^{2}$-cone reducible if $r$ is ${\cal C}^{2}$-cone reducible at every point $x\in\operatorname{dom}r$. 
\end{definition}

Note that the class of ${\cal C}^2$-cone reducible functions includes, in particular, the $\ell_{p}$ norm functions (where $p=1,2,\infty$), the nuclear norm function, the matrix spectral norm function,  the indicator functions of the nonnegative orthant cone, the second-order cone, and the positive-semidefinite cone. Therefore, all of these functions satisfy Assumption~\ref{assump-1}.\vspace*{0.05in} 

Prior to the formulation of the next assumption, recall that 
the range of a multifunction $S\colon\mathcal{R}^n\tto\mathcal{R}^m$ is the set $\operatorname{rge}S:=\{u\in\mathcal{R}^{m}\,|\,\exists\,x\in\mathcal{R}^{n}\, \mbox{ with }\, u\in S(x)\}$.

\begin{assumption}\label{assump-2}
Let $r:\mathcal{R}^{m}\rightarrow \overline{\mathcal{R}}$ be an l.s.c.\ proper convex function with $\bar{x}\in\operatorname{dom}r$ and $\bar{u}\in\partial r(\bar{x})$. Suppose that there exists $\overline{W}\in\Bar\grad\operatorname{Prox}_{r}(\bar{x}+\bar{u})$ from the limiting Jacobian \eqref{Bsub} of the proximal mapping \eqref{prox} ensuring the range representation
\begin{align*}
\bigcup\limits_{W\in J\operatorname{Prox}_{r}(\bar{x}+\bar{u})}\operatorname{rge}W=\operatorname{rge}\overline{W},
\end{align*}
and for any $y\in\bigcup\limits_{W\in J\operatorname{Prox}_{r}(\bar{x}+\bar{u})}\operatorname{rge}W$ we  have $y=\overline{W}\bar{z}$, where $\bar{z}$ satisfies the equality
\begin{align*}
\min\limits_{{z\in\mathcal{R}^{m},y=Wz}\atop{W\in J\operatorname{Prox}_{r}(\bar{x}+\bar{u})}}\langle y,z-y\rangle=\langle y,\bar{z}-y\rangle.
\end{align*}
\end{assumption}

Assumption~\ref{assump-2} is more technical compared to Assumption~\ref{assump-1} while allowing us to describe a general class of COPs for which strong variational sufficiency can be {\em fully characterized} by pointbased second-order conditions. The fulfillment of this assumption can be readily verified if the matrix set $J\operatorname{Prox}_{r}(\bar{x}+\bar{u})$ is available. In Appendix~\ref{appendix-A}, we show that Assumption~\ref{assump-2} is satisfied for the highly important nonpolyhedral cases of the indicator function associated with positive-semidefinite cones and of the nuclear norm function. \vspace*{-0.1in}

\section{Limiting Coderivative of Proximal Mappings}\label{sec:prox}\vspace*{-0.05in}

In this section, we evaluate and estimate the limiting coderivative \eqref{lim-cod} of the proximal mapping \eqref{prox}. The results obtained below are of their own interest while being important to establish the main second-order characterizations of strong variational sufficiency in (COP).\vspace*{0.05in}

Recall first the following useful rule proved in  \cite[Theorem~1.17]{Mordukhovich2006} in general Banach spaces. 

\begin{lemma}\label{lem:normal-cone-trans}
Let $S:\mathcal{R}^{m}\rightarrow\mathcal{R}^{n}$ be a mapping strictly differentiable at $\bar{x}$ and such that the derivative operator $S'(\bar{x})$ is surjective, and let $\Theta\subset\mathcal{R}^{n}$ with $\bar{y}=S(\bar{x})\in\Theta$. Then we have 
\begin{align*}
\mathcal{N}_{S^{-1}(\Theta)}(\bar{x})=S'(\bar{x})^{T}\mathcal{N}_{\Theta}(\bar{y}).
\end{align*}
\end{lemma}

Now we obtain a relationship between the limiting coderivative and generalized Jacobian of the proximal mapping associated with an extended-real-valued function. 

\begin{proposition}\label{prop:relation-coderivative-element}
Let $r:\mathcal{R}^{m}\rightarrow\Bar{\mathcal R}$ be an l.s.c.\  convex function with  $\bar{x}\in\operatorname{dom}r$, $\bar{u}\in\partial r(\bar{x})$. Given $\sigma>0$, for any $d\in\mathcal{R}^{m}$ and $v\in D^{*}\operatorname{Prox}_{\sigma r^{*}}(\bar{u}+\sigma\bar{x})(d)$, there is $U\in J\operatorname{Prox}_{r}(\bar{x}+\bar{u})$ with
\begin{align*}
v=(I+(\sigma-1)U)^{-1}(I-U)d.
\end{align*}
\end{proposition}
\begin{proof}
Deduce from the definition of the limiting coderivative and the relationships in \eqref{eq:inverse-graphical-derivative} and \eqref{eq:coderivative-sum} that the inclusion $v\in D^{*}\operatorname{Prox}_{\sigma r^{*}}(\bar{u}+\sigma\bar{x})(d)$ is equivalent to 
\begin{align*}
-d\in D^{*}(I+\sigma\partial r^{*})(\bar{u},\bar{u}+\sigma\bar{x})(-v)\Longleftrightarrow -d+v\in D^{*}(\sigma \partial r^{*})(\bar{u},\sigma\bar{x})(-v).
\end{align*}
It is easy to observe the representation
\begin{align*}
\operatorname{gph}\partial r^{*}=\left[\begin{array}{cc}
I&0\\0&\sigma^{-1} I
\end{array}\right]\operatorname{gph}\sigma\partial r^{*}.
\end{align*}
Applying Lemma~\ref{lem:normal-cone-trans}, we get
\begin{align*}
\mathcal{N}_{\operatorname{gph}\partial r^{*}}(\bar{u},\bar{x})=\left[\begin{array}{cc}
I&0\\0&\sigma I
\end{array}\right]\mathcal{N}_{\operatorname{gph}\sigma\partial r^{*}}(\bar{u},\sigma\bar{x}).
\end{align*}
Combining \eqref{eq:inverse-graphical-derivative}, this leads us to the equivalencies
\begin{align*}
-d+v\in D^{*}(\sigma \partial r^{*})(\bar{u},\sigma\bar{x})(-v)&\Longleftrightarrow-d+v\in D^{*}(\partial r^{*})(\bar{u},\bar{x})(-\sigma v)\Longleftrightarrow \sigma v\in D^{*}(\partial r)(\bar{x},\bar{u})(d-v),
\end{align*}
which imply in turn by using \eqref{eq:inverse-graphical-derivative} and \eqref{eq:coderivative-sum} that
\begin{align*}
v-d\in D^{*}\operatorname{Prox}_{r}(\bar{x}+\bar{u})(-d-(\sigma-1)v).
\end{align*}
Finally, it follows from \eqref{cod-cl} that there exists a matrix $U\in J\operatorname{Prox}_{r}(\bar{x}+\bar{u})$ satisfying 
\begin{align*}
v-d=U(-d-(\sigma-1)v).
\end{align*}
This clearly yields the claimed result and completes the proof of the proposition.
\end{proof}

Next we derive the main result of this section providing a novel explicit estimate of the limiting coderivative of the proximal mapping via the limiting Jacobian \eqref{Bsub}.

\begin{theorem}\label{prop-ineq-norm-W} Let $r:\mathcal{R}^{m}\rightarrow \overline{\mathcal{R}}$, $\bar{x}$, and $\bar{u}$ be the same as in 
Proposition~{\rm\ref{prop:relation-coderivative-element}}, and let $\sigma>0$. Suppose in addition  that the function $r$ satisfies Assumption~{\rm\ref{assump-2}} at $\bar{x}$ for $\bar{u}$ with the matrix $\overline{W}\in\Bar\grad\operatorname{Prox}_{r}(\bar{x}+\bar{u})$ therein. Then for any $d\in\mathcal{R}^{m}$ and $z\in D^{*}\operatorname{Prox}_{\sigma r^{*}}(\bar{u}+\sigma\bar{x})(d)$ satisfying $z=(I+(\sigma-1)U)^{-1}(I-U)d$ and $U\in J\operatorname{Prox}_{r}(\bar{x}+\bar{u})$, we have
\begin{align}\label{cod-est}
\|z\|^{2}\geq\|(I-UU^{\dagger})d\|^{2}\geq\|(I-\overline{W}\overline{W}^{\dagger})d\|^{2}.
\end{align} 
Moreover, we also have
\begin{align}\label{sigma-cod-est}
    \|\sigma(I+(\sigma-1)U)^{-1}Ud-UU^{\dagger}d\|\rightarrow0,\;\mbox{as}\;\sigma\rightarrow\infty.
\end{align}
\end{theorem}
\begin{proof}
Picking any $d\in\mathcal{R}^{m}$ and $z\in D^{*}\operatorname{Prox}_{\sigma r^{*}}(\bar{u}+\sigma\bar{x})(d)$, deduce from Proposition~\ref{prop:relation-coderivative-element} the existence of a matrix $U\in J\operatorname{Prox}_{r}(\bar{x}+\bar{u})$ with rank $0\le l\le m$ such that $z=(I+(\sigma-1)U)^{-1}(I-U)d$. Since $U$ is a positive-semidefinite matrix, its singular value decomposition can be expressed via an orthogonal matrix $R$ and $S_{l}=\operatorname{Diag}\{\lambda_{1},\ldots,\lambda_{l}\}$ with $\lambda_{i}>0$ $(i=1,\ldots,l)$ as
\begin{align*}
U=R\left[\begin{array}{cc}
S_{l}&0\\0&0
\end{array}\right]R^{T}.
\end{align*} 
Consequently, the Moore-Penrose inverse $U^{\dagger}$ is given by
\begin{align*}
U^{\dagger}=R\left[\begin{array}{cc}
S^{-1}_{l}&0\\0&0
\end{array}\right]R^{T}.
\end{align*}
Using the above tells us that
\begin{align*}
\|(I-UU^{\dagger})d\|^{2}&=\left\langle(I-UU^{\dagger})d,(I-UU^{\dagger})d\right\rangle=\left\langle R^{T}d,\left[\begin{array}{cc}
0&0\\0&I_{m-l}
\end{array}\right]R^{T}d\right\rangle,\\
(I+(\sigma-1)U)^{-1}(I-U)&=R\left[\begin{array}{cc}
(I_{l}+(\sigma-1)S_{l})^{-1}&0\\0&I_{m-l}
\end{array}\right]\left[\begin{array}{cc}
I_{l}-S_{l}&0\\0&I_{m-l}
\end{array}\right]R^{T}\\
&=R\left[\begin{array}{cc}
(I_{l}+(\sigma-1)S_{l})^{-1}(I_{l}-S_{l})&0\\0&I_{m-l}
\end{array}\right]R^{T},\\
\sigma(I+(\sigma-1)U)^{-1}U-UU^{\dagger}&=R\left[\begin{array}{cc}
     -(I_{l}+(\sigma-1)S_{l})^{-1}(I_{l}-S_{l})& 0 \\
     0 & 0
\end{array}\right]R^{T},
\end{align*}
which leads us therefore to the relationships
\begin{align}\label{eq:coderivative-norm-ineq}
\begin{aligned}
\langle z,z\rangle&=\left\langle R^{T}d,\left[\begin{array}{cc}
(I_{l}+(\sigma-1)S_{l})^{-1}(I_{l}-S_{l})&0\\0&I_{m-l}
\end{array}\right]^{2}R^{T}d\right\rangle\\
&\geq\left\langle R^{T}d,\left[\begin{array}{cc}
0&0\\0&I_{m-l}
\end{array}\right]R^{T}d\right\rangle=\|(I-UU^{\dagger})d\|^{2},
\end{aligned}
\end{align}
and $\|\sigma(I+(\sigma-1)U)^{-1}Ud-UU^{\dagger}d\|\rightarrow0$ as $\sigma\rightarrow\infty$.
It follows from Assumption~\ref{assump-2} that $UU^{\dagger}d\subset\operatorname{rge}U\subset\operatorname{rge}\overline{W}$. Observing that $\overline{W}\overline{W}^{\dagger}d\in\operatorname{rge}\operatorname{\overline{W}}$ yields $\overline{W}\overline{W}^{\dagger}d-UU^{\dagger}d\in\operatorname{rge}\overline{W}$. Since $(I-\overline{W}\overline{W}^{\dagger})d\in\operatorname{ker}(\overline{W})=\operatorname{ker}(\overline{W}^{T})$, we have
\begin{align*}
\left\langle(I-\overline{W}\overline{W}^{\dagger})d,\overline{W}\overline{W}^{\dagger}d-UU^{\dagger}d\right\rangle=0.
\end{align*}
This allows us to deduce from \eqref{eq:coderivative-norm-ineq} that
\begin{align*}
&\langle z,z\rangle\geq\|(I-UU^{\dagger})d\|^{2}=\|(I-\overline{W}\overline{W}^{\dagger})d+(\overline{W}\overline{W}^{\dagger}d-UU^{\dagger}d)\|^{2}\\
&\qquad\quad=\|(I-\overline{W}\overline{W}^{\dagger})d\|^{2}+\|\overline{W}\overline{W}^{\dagger}d-UU^{\dagger}d\|^{2}\geq\|(I-\overline{W}\overline{W}^{\dagger})d\|^{2},
\end{align*}
which derives \eqref{cod-est} and thus completes the proof of the theorem.
\end{proof}

\section{The Second-Order Variational Function and SSOSC}\label{sec:ssosc}

In this section, we introduce a novel notion of the second-order variational function and formulate in its terms an extended version of the SSOSC for a general class of composite optimization problems of type (COP). 

We begin with defining the {\em Lagrangian function} associated with (COP) by
\begin{align*}
\mathcal{L}(x,u):=\inf_{s\in\mathcal{R}^{m}}\Big\{f_{0}(x)+g(F(x)+s)-\langle u,s\rangle\Big\}=L(x,u)-g^{*}(u),
\end{align*}
where $L(x,u):=f_{0}(x)+\langle F(x),u\rangle$. The corresponding {\em Karush-Kuhn-Tucker} (KKT) conditions for (COP) form the following system:
\begin{align}\label{eq:kkt-problem-P}
\nabla_{x}L(x,u)=0,\quad u\in\partial g(F(x)).
\end{align} 
For a given $\sigma>0$, the {\em augmented Lagrangian function} of (COP) is defined by
\begin{align*}
\mathcal{L}_{\sigma}(x,u):=\sup_{s\in\mathcal{R}^{m}}\Big\{\mathcal{L}(x,s)-\frac{1}{2\sigma}\|s-u\|^{2}\Big\}=f_{0}(x)-e_{\sigma g^{*}}(u+\sigma F(x))+\langle u,F(x)\rangle+\frac{\sigma}{2}\|F(x)\|^{2}.
\end{align*}
The {\em generalized Hessian} of $\mathcal{L}_{\sigma}(x,u)$ with respect to the variable $x$ is given by
\begin{align*}
\Bar{\partial}_{xx}^{2}\mathcal{L}_{\sigma}(x,u):=\nabla_{xx}^{2}L(x,u)+\sigma F'(x)^{T} J\operatorname{Prox}_{\sigma g^{*}}(u+\sigma F(x))F'(x).
\end{align*}
For any $U\in J\operatorname{Prox}_{\sigma g^{*}}(u+\sigma F(x))$, observe that 
\begin{align*}
\nabla_{xx}^{2}L(x,u)+\sigma F'(x)^{T}UF'(x)\in\Bar{\partial}_{xx}^{2}\mathcal{L}_{\sigma}(x,u).
\end{align*}
It follows from the chain rule in \cite[Proposition~2.3.3 and Theorem~2.6.6]{Clarke1990} that 
\begin{align*}
J\nabla_{x}\mathcal{L}_{\sigma}(x,u)(d)\subset\Bar{\partial}_{xx}^{2}\mathcal{L}_{\sigma}(x,u)(d)\;\mbox{ for any }\;d\in\mathcal{R}^{n},
\end{align*}
which yields the positive-definiteness of all the matrices from $J\nabla_{x}\mathcal{L}_{\sigma}(x,u)$ provided that those in $\Bar{\partial}_{xx}^{2}\mathcal{L}_{\sigma}(x,u)$ have this property.

Given further an extended real-valued function $r:\mathcal{R}^{m}\rightarrow \overline{\mathcal{R}}$ with $x\in\operatorname{dom}r$ and $u\in\partial r(x)$, let us introduce the {\em second-order variational function} (SOVF) $\Gamma_{r}(x,u):\mathcal{R}^{m}\rightarrow\overline{\mathcal{R}}$ by
\begin{align}\label{sovf}
\Gamma_{r}(x,u)(v):=\left\{\begin{array}{cl}\min\limits_{{d\in\mathcal{R}^{m},v=Wd}\atop{W\in J\operatorname{Prox}_{r}(x+u)}}\langle v,d-v\rangle & \mbox{if}\ v\in\bigcup\limits_{W\in J\operatorname{Prox}_{r}(x+u)}\operatorname{rge}W,\\
\infty&\mbox{otherwise}.
\end{array}\right.
\end{align}
Using \eqref{sovf} leads us to formulating the SSOSC for (COP).

\begin{definition}\label{def:ssosc} Let $(\bar{x},\bar{u})$ be a solution to the KKT system \eqref{eq:kkt-problem-P}. We say that the {\sc strong second-order sufficient condition} $($SSOSC$)$ holds at $\bar{x}$ for $\bar{u}$ if
\begin{align}\label{eq:SSOSC}
\left\langle \nabla_{xx}^{2}L(\bar{x},\bar{u})d,d\right\rangle+\Gamma_{g}(F(\bar{x}),\bar{u})(F'(\bar{x})d)>0\;\mbox{ for all }\;d\ne 0.
\end{align}
\end{definition}

Recall the {\em nuclear norm function} $\|\cdot\|_{*}$ for a given matrix is the sum of all its singular values. The importance of this function in a variety of applications to physics, engineering, computational mathematics, etc. has been well recognized. The example below provides an explicit calculation of the SSOSC for COPs, where $g$ is the nuclear norm function. 

\begin{example} {\rm 
Let $p\leq q$ be two positive integers, and let $g$ be the nuclear norm function $\|\cdot\|_{*}$. Given $(\ox,\bar u)$ satisfying \eqref{eq:kkt-problem-P}, define $X:=F(\bar{x})$, $U:=\bar{u}$, and $Y:=F'(\bar{x})d$. Letting $A=X+U$, denote its singular value decomposition by $A=R\,[\Sigma\ 0]\,S^{T}$, where $R\in\mathcal{R}^{p\times p}$ and $S\in\mathcal{R}^{q\times q}$ are orthogonal matrices, $\Sigma:=\operatorname{Diag}(\sigma_{1},\ldots,\sigma_{p})$ is the diagonal matrix of singular values of $A$ with $\sigma_{1}\geq\cdots\geq\sigma_{p}\ge 0$, and $\widetilde{Y}=R^{T}YS$. Consider the index sets 
\begin{align*}
\alpha:=\{1,\ldots,p\},\
\alpha_{1}:=\{i\in\alpha\,|\,\sigma_{i}>1\},\ \alpha_{2}:=\{i\in\alpha\,|\,\sigma_{i}=1\},\ \alpha_{3}:=\{i\in\alpha\,
|\,\sigma_{i}<1\}.
\end{align*} 
Using the calculations in Example~\ref{example-nuclear} of
Appendix~\ref{appendix-A}, the SSOSC in \eqref{eq:SSOSC} is expressed as follows:
\begin{align*}
&\left\langle \nabla_{xx}^{2}L(\bar{x},\bar{u})d,d\right\rangle+2\sum_{{i<j,i\in\alpha_{1}}\atop{j\in\alpha_{1}\cup\alpha_{2}}}\left(\frac{\sigma_{i}+\sigma_{j}}{\sigma_{i}-1+\max\{\sigma_{j}-1,0\}}-1\right)\left(\frac{\widetilde{Y}_{ij}-\widetilde{Y}_{ji}}{2}\right)^{2}\\
&\ +2\sum_{{i\in\alpha_{1}}\atop{j\in\alpha_{3}}}\left[\frac{1-\sigma_{j}}{\sigma_{i}-1}\left(\frac{\widetilde{Y}_{ij}+\widetilde{Y}_{ji}}{2}\right)^{2}+\frac{1+\sigma_{j}}{\sigma_{i}-1}\left(\frac{\widetilde{Y}_{ij}-\widetilde{Y}_{ji}}{2}\right)^{2}\right]+\sum_{i\in\alpha_{1}}\frac{1}{\sigma_{i}-1}\sum_{j=p+1}^{q}\widetilde{Y}_{ij}^{2}>0,\\ 
&\qquad\qquad\qquad\qquad\qquad\qquad\qquad\qquad\mbox{for all }\;F'(\bar{x})d\in\bigcup\limits_{W\in J\operatorname{Prox}_{\|\cdot\|_{*}}(X+U)}\operatorname{rge}W,\;d\ne 0.
\end{align*}
The explicit expression for $\bigcup\limits_{W\in J\operatorname{Prox}_{\|\cdot\|_{*}}(X+U)}\operatorname{rge}W$ is given in Example~\ref{example-nuclear} of Appendix~\ref{appendix-A}.}
\end{example}

Under the fulfillment of Assumption~\ref{assump-2}, the second-order variational function \eqref{sovf} can be equivalently expressed  at the KKT point $(x,u)$ as
\begin{align}\label{eq:equivalent-form-Gamma-g}
\Gamma_{g}(F(x),u)(v)=\min_{-d\in D^{*}\partial g(F(x),u)(-v)}\langle v,d\rangle
\end{align}	
(see \cite[Lemma~5.3]{Tangwang2024} for more details). 
The function $\Gamma_{g}(F(x),u)$ is further simplified to 
\begin{align}\label{sovf1}
\Gamma_{g}(F(x),u)(v)=\left\langle v,\left(\overline{W}^{\dagger}-I\right)v\right\rangle:=\Upsilon_{g}(F(x),u)(v).
\end{align}
Combining this with the identity
\begin{align*}
\bigcup\limits_{W\in J\operatorname{Prox}_{g}(F(x)+u)}\operatorname{rge}W=\operatorname{ker}\left(I-\overline{W}\overline{W}^{\dagger}\right),
\end{align*} 
the SSOSC \eqref{eq:SSOSC} can be reformulated as
\begin{align*}
\left\langle \nabla_{xx}^{2}L(x,u)d,d\right\rangle+\Upsilon_{g}(F(x),u)(F'(x)d)>0\;\mbox{ for all }\;F'(x)d\in\operatorname{ker}\left(I-\overline{W}\overline{W}^{\dagger}\right),\,d\ne 0.
\end{align*}

To conclude this section, we present a pivotal result from \cite{Rockafellar2023} that is deeply intertwined with strong variational sufficiency for local optimality. Let us first recall the definitions.

\begin{definition}
A function $q:\mathcal{R}^{m}\rightarrow\overline{\mathcal{R}}$ is a {\sc generalized quadratic form} on $\mathcal{R}^{m}$ if $q(0)=0$ and the subgradient mapping $\partial q$ is generalized linear, i.e., $\operatorname{gph}\partial q$ is a subspace of $\mathcal{R}^{m}\times\mathcal{R}^{m}$. A function $r:\mathcal{R}^{m}\rightarrow \overline{\mathcal{R}}$ is {\sc generalized twice differentiable} at $\bar{x}\in\operatorname{dom}r$ for $\bar{u}\in\partial r(\bar{x})$ if it is twice epi-differentiable at $\bar{x}$ for $\bar{u}$ with the second-order subderivative $d^{2}r(\bar{x},\bar{u})$ being a generalized quadratic form on $\mathcal{R}^{m}$.
\end{definition}

\begin{definition}
Let $r:\mathcal{R}^{m}\rightarrow\overline{R}$ be an l.s.c. proper function with $\bar{x}\in\operatorname{dom}r$ and $\bar{u}\in\partial r(\bar{x})$. The {\sc quadratic bundle} of $r$ at $\bar{x}$ for $\bar{u}$ is defined by
\begin{align*}
\operatorname{quad}r(\bar{x},\bar{u}):=\left[\begin{array}{l}\mbox{the collection of generalized quadratic forms}\ q \mbox{ for which there exist } \\(x^{k},u^{k})
\rightarrow(\bar{x},\bar{u}) \mbox{ with $r$ generalized twice differentiable
at } x^{k} \mbox{ for } u^{k}\\
\mbox{ and such that the generalized quadratic
forms } q^{k}=\frac{1}{2}d^{2}r(x^{k},u^{k})\\ \mbox{ converge epigraphically to } q.   \end{array} \right.
\end{align*}
\end{definition}

Quadratic bundles play an important role in characterizing the variational convexity of functions. The recent papers \cite{Helmut2025,KhanMordukhovich2025} systematically investigate the variational convexity of prox-regular functions using quadratic bundles, particularly in settings where subdifferential continuity is absent.
It is proved in \cite[Theorem~5]{Rockafellar2023} that strong variational sufficiency can be fully characterized by the quadratic bundle.
We formulate this result for further references.

\begin{theorem}\label{thm:quad-bundle}
Let $(\bar{x},\bar{u})$ be a solution to the KKT system \eqref{eq:kkt-problem-P}. The strong variational sufficient condition for the local optimality of $\ox$  in $(P)$ relative to $\bar u$ holds if and only if for every $q\in\operatorname{quad}g(F(\bar{x}),\bar{u})$ we have the strict inequality
\begin{align*}
\left\langle \nabla_{xx}^{2}L(\bar{x},\bar{u})d,d\right\rangle+q(F'(\bar{x})d)>0\;\mbox{ whenever }\;d\in\mathcal{R}^{n}\setminus\{0\}.
\end{align*}
\end{theorem}

\section{Characterizations of Strong Variational Sufficiency}\label{sec:char}

In this section, we establish pointbased second-order characterizations of strong variational sufficiency for local optimality in problem $(P)$ associated with a general class of COPs under Assumptions~\ref{assump-1} and \ref{assump-2}. Crucial building blocks of our analysis are given in the following collections of lemmas, which are of their own interest.

\begin{lemma}\label{lem:postive-homogeneous-property}
Let $l:\mathcal{R}^{m}\rightarrow\mathcal{R}$ be an l.s.c.\ function that is  positively homogeneous of degree two. Given a linear operator $M:\mathcal{R}^{m}\rightarrow\mathcal{R}^{n}$, assume that there is $\bar{\kappa}>0$ such that for any $d$ satisfying $Md=0$ we have $l(d)\geq\bar{\kappa}\|d\|^{2}$. Then there exist numbers $\kappa\in(0,\bar{\kappa}]$ and $c>0$ ensuring that
\begin{align*}
l(d)+c\|Md\|^{2}\geq\kappa\|d\|^{2}\;\mbox{ whenever }\;d\in\mathcal{R}^{m}.
\end{align*}
\end{lemma}
\begin{proof}
Since $l$ is positively homogeneous of degree two with $l(0)\geq0$, it suffices to verify the claimed inequality for all $d$ satisfying $Md\neq 0$ and $\|d\|=1$. Specifically, we aim to show that
\begin{align*}
l(d)+c\|Md\|^{2}\geq\kappa\|d\|^{2}\;\mbox{ whenever }\; Md\neq 0,\, \|d\|=1.
\end{align*}
Consider the function $\theta(d):=\frac{l(d)-\bar{\kappa}\|d\|^{2}}{\|Md\|^{2}}$ on the compact set $\{d\,|\,\|d\|=1,\,Md\neq 0\}$. This function is l.s.c.\ and tends to $\infty$ as $d$ approaches any boundary point of the set. By the classical Weierstrass theorem, $\theta$ attains its minimum denoted by $c^{*}$. Consequently, the claimed estimate is satisfied for any number $c\geq-c^{*}$.
\end{proof}

\begin{lemma}\label{lem-domain-second-derivative}
Let $r:\mathcal{R}^{m}\rightarrow \overline{\mathcal{R}}$ with $\bar{x}\in\operatorname{dom}r$ and $\bar{u}\in\partial r(\bar{x})$ satisfy Assumption~{\rm\ref{assump-1}}. Then 
\begin{align}\label{eq:domain-second-order-subderivative}
\operatorname{dom}d^{2}r(\bar{x},\bar{u})=\left\{d\,\left|\,dr(\bar{x})(d)=\langle \bar{u},d\rangle\right.\right\}=\operatorname{dom}D(\partial r)(\bar{x},\bar{u}).
\end{align}
\end{lemma}	
\begin{proof}
The first equality in \eqref{eq:domain-second-order-subderivative} follows from \cite[Proposition~3.8]{Mohammadi2020}. To proceed further, deduce from
\cite[Proposition~13.20]{Rockafellar1998}, that $d^{2}r(\bar{x},\bar{u})$ is convex and is expressed as
\begin{align*}
d^{2}r(\bar{x},\bar{u})=\gamma_{C}^{2}
\end{align*} 
via the gauge of $C=\{v\,|\,d^{2}r(\bar{x},\bar{u})(v)\leq 1\}$, which is a closed convex set containing the origin. Since $\gamma_{C}$ is a positively homogeneous convex function with $\gamma_{C}(0)=0$, it follows from \cite[Proposition~2.124]{Bonnans2000} that it is subdifferentiable at any $v \in \operatorname{dom} d^{2}r(\bar{x},\bar{u})$. Specifically, for every such $v$ which is not a minimizer of $d^{2}r(\bar{x},\bar{u})$, there exists a subgradient $w \in \partial\gamma_{C}(v)$ satisfying
\begin{align*}
\gamma_{C}(y)\geq\gamma_{C}(v)+\langle w,y-v\rangle.
\end{align*}
Applying \cite[Propositions~13.5 and 13.20]{Rockafellar1998} tells us that
\begin{align*}
d^{2}r(\bar{x},\bar{u})(v)=\gamma^{2}_{C}(v)>0
\end{align*}
and there exists a neighborhood $\mathcal{O}$ of $v$ such that for any $y \in \mathcal{O}$ we get $d^{2}r(\bar{x},\bar{u})(y) > 0$ and
\begin{align*}
\gamma_{C}(v)+\langle w,y-v\rangle>0.
\end{align*} 
This implies therefore the relationships
\begin{align*}
d^{2}r(\bar{x},\bar{u})(y)=\gamma_{C}^{2}(y)&\geq\gamma_{C}^{2}(v)+2\gamma_{C}(v)\langle w,y-v\rangle+\langle w,y-v\rangle^{2}\\
&\geq d^{2}r(\bar{x},\bar{u})(v)-2\gamma_{C}(v)\|w\|\|y-v\|.
\end{align*}
It follows from \cite[Proposition~2.1]{Mohammadi2022} that $\partial d^{2}r(\bar{x},\bar{u})(v) \neq \emptyset$. Conversely, if $v \in \mathop{\arg\min}\limits_{x\in\mathcal{R}^{m}} {d^{2}r(\bar{x},\bar{u})(x)}$, then $0 \in \partial d^{2}r(\bar{x},\bar{u})(v)$ and the fulfillment of \eqref{eq:domain-second-order-subderivative} is a consequence of \cite[Theorem~13.40]{Rockafellar1998}.
\end{proof}

\begin{lemma}\label{lemma-graphical-second-order-derivative}
Let $r:\mathcal{R}^{m}\rightarrow \overline{\mathcal{R}}$ with $\bar{x}\in\operatorname{dom}r$ and $\bar{u}\in\partial r(\bar{x})$ satisfy Assumption~{\rm\ref{assump-1}}. Given any $v,d\in\mathcal{R}^{m}$, we have that $v\in D\operatorname{Prox}_{r}(\bar{x}+\bar{u})(v+d)$ holds if and only if $d\in D(\partial r)(\bar{x},\bar{u})(v)$. Furthermore, $d^{2}r(\bar{x},\bar{u})(v)=\langle v,d\rangle$ whenever $d\in D(\partial r)(\bar{x},\bar{u})(v)$,
\end{lemma}
\begin{proof}
We get by \cite[Proposition~3.8]{Mohammadi2020} that $r$ is twice epi-differentiable at $\bar{x}$ for $\bar{u}$. Since $\operatorname{Prox}_{r}(\cdot)=(I+\partial r)^{-1}(\cdot)$, it follows from \eqref{eq:inverse-graphical-derivative} and \eqref{eq:graphical-sum} the equivalence
$$
v\in D\operatorname{Prox}_{r}(\bar{x}+\bar{u})(v+d)\Longleftrightarrow d\in D(\partial r)(\bar{x},\bar{u})(v).
$$
Applying then \cite[Theorem~13.21]{Rockafellar1998} and \cite[Theorem~23.5]{Rockafellar1970} leads us to
\begin{align*}
\langle v,d\rangle=\frac{1}{2}d^{2}r(\bar{x},\bar{u})(v)+\frac{1}{2}d^{2}r^{*}(\bar{u},\bar{x})(d).
\end{align*}
Finally, it follows from \cite[Corollary~15.3.2]{Rockafellar1970} that
\begin{align*}
\langle v,d\rangle\leq\sqrt{d^{2}r(\bar{x},\bar{u})(v)d^{2}r^{*}(\bar{u},\bar{x})(d)},
\end{align*}
which verifies $d^{2}r(\bar{x},\bar{u})(v)=\langle v,d\rangle$ and thus completes the proof.
\end{proof}

\begin{lemma}\label{lem-quadratic-budle-Upsilon}
Let $r:\mathcal{R}^{m}\rightarrow \overline{\mathcal{R}}$ with $\bar{x}\in\operatorname{dom}r$ and $\bar{u}\in\partial r(\bar{x})$ satisfy Assumption~{\rm\ref{assump-1}} at $x$ for all $u\in\partial r(x)$ with $x$ in some neighborhood of $\bar{x}$ and Assumption~{\rm\ref{assump-2}} at $\bar{x}$ for $\bar{u}$, and let $\operatorname{quad}r(\bar{x},\bar{u})$ be the corresponding quadratic bundle. Then there exists a sequence $\{(x^{k},u^{k})\}$ such that $d^{2}r(x^{k},u^{k})$ epi-converges to $\Gamma_{r}(\bar{x},\bar{u})$ as $k\to\infty$, and $\Gamma_{r}(\bar{x},\bar{u})\in\operatorname{quad}r(\bar{x},\bar{u})$.
\end{lemma}
\begin{proof} We start with deriving the inclusion
\begin{align}\label{quad-inc}
\operatorname{rge}\widehat{D}^{*}\operatorname{Prox}_{r}(x+u)\subset\left\{d\,\left|\,0\in D\operatorname{Prox}_{r}(x+u)(d)\right.\right\}^{\circ}=\operatorname{rge}D\operatorname{Prox}_{r}(x+u)\
\end{align}
for all $x\in\operatorname{dom}r$ and $u\in\partial r(x)$. Picking any $y\in\left\{d\,\left|\,0\in D\operatorname{Prox}_{r}(x+u)(d)\right.\right\}$ with $(y,0)\in\mathcal{T}_{\operatorname{gph}\operatorname{Prox}_{r}}(x+u)$ and taking a pair $(w,z)\in\widehat{\mathcal{N}}_{\operatorname{gph}\operatorname{Prox}_{r}}(x+u)$ with $w\in\operatorname{rge}\widehat{D}^{*}\operatorname{Prox}_{r}(x+u)$ yield $\langle y,w\rangle=\langle(y,0),(w,z)\rangle\leq 0$. This ensures therefore that
\begin{align*}
w\in\left\{d\,\left|\,0\in D\operatorname{Prox}_{r}(x+u)(d)\right.\right\}^{\circ}.
\end{align*} 
It follows from Lemma~\ref{lemma-graphical-second-order-derivative} that 
\begin{align*}
\left\{d\,\left|\,0\in D\operatorname{Prox}_{r}(x+u)(d)\right.\right\}=D(\partial r)(x,u)(0).
\end{align*}	
Moreover, we deduce from \cite[Proposition~3.3]{Tangwang2024} that
\begin{align}\label{eq:tangent-cone-graphical}
D(\partial r)(x,u)(0)=\Big\{d\,\Big|\,dr(x)(d)=\langle u,d\rangle\Big\}^{\circ}.
\end{align}
Taking the polar in both sides of \eqref{eq:tangent-cone-graphical} and combining Lemma~\ref{lem-domain-second-derivative} and Lemma~\ref{lemma-graphical-second-order-derivative}, we conclude that for $x\in\operatorname{dom}r$ and $u\in\partial r(x)$, it holds
\begin{align*}
\begin{array}{ll}
\operatorname{dom}D(\partial r)(x,u)&=\operatorname{rge}D\operatorname{Prox}_{r}(x+u)\\
&=\Big\{d\,\Big|\,dr(x)(d)=\langle u,d\rangle\Big\}\supset\operatorname{rge}\widehat{D}^{*}\operatorname{Prox}_{r}(x+u).
\end{array}
\end{align*}  
Therefore, for any $v\in\operatorname{rge}\widehat{D}^{*}\operatorname{Prox}_{r}(x+u)$, there exist $d,d'\in\mathcal{R}^{m}$ such that 
\begin{align*}
v\in\widehat{D}^{*}\operatorname{Prox}_{r}(x+u)(d),\ v\in D\operatorname{Prox}_{r}(x+u)(d'), 
\end{align*}
which readily verifies the claimed inclusion in \eqref{quad-inc}.

Fix $\xi\in D^{*}\operatorname{Prox}_{r}(\bar{x}+\bar{u})(\eta)$ and find by definition \eqref{lim-cod} sequences $(x^{k},u^{k})$ and $(\xi^{k},\eta^{k})$ with $\xi^{k}\in\widehat{D}^{*}\operatorname{Prox}_{r}(x^{k}+u^{k})(\eta^{k})$ such that $(x^{k},u^{k})\rightarrow(x,u)$ and $(\xi^{k},\eta^{k})\rightarrow(\xi,\eta)$. By \eqref{quad-inc}, there exist directions $d^{k}\in\mathcal{R}^{m}$ satisfying the equalities
\begin{align*}
\xi^{k}=D\operatorname{Prox}_{r}(x^{k}+u^{k})(d^{k})
\end{align*}
for all $k\in\mathbb N$. Using now Lemma~\ref{lemma-graphical-second-order-derivative} and \cite[Proposition~8.37]{Rockafellar1998} tells us that 
\begin{align*}
d^{2}r(x^{k},u^{k})(\xi^{k})=\langle\xi^{k},d^{k}-\xi^{k}\rangle\leq\langle\xi^{k},\eta^{k}-\xi^{k}\rangle.
\end{align*}
Passing to the limit as $k\rightarrow\infty$ and using the SOVF representation \eqref{eq:equivalent-form-Gamma-g} yield
\begin{align*}
\langle\xi^{k},\eta^{k}-\xi^{k}\rangle\rightarrow\langle\xi,\eta-\xi\rangle=\Gamma_{r}(\bar{x},\bar{u})(\xi),
\end{align*}
which implies in turn that 
\begin{align*}
\operatorname{epi}\Gamma_{r}(\bar{x},\bar{u})\subset\lim_{k\rightarrow\infty}\operatorname{epi}d^{2}r(x^{k},u^{k}).
\end{align*}
Take any sequence $(x^{k},u^{k})$ such that $u^{k}\in\partial r(x^{k})$, $(x^{k},u^{k})\rightarrow(x,u)$, and $\operatorname{Prox}_{r}$ is differentiable at $x^{k}+u^{k}$ with $\operatorname{Prox}'_{r}(x^{k}+u^{k})\rightarrow\overline{W}$. Given $d\in\bigcup\limits_{W\in J\operatorname{Prox}_{r}(\bar{x}+\bar{u})}\operatorname{rge}W$, we get
\begin{align*}
D\operatorname{Prox}_{r}(x^{k}+u^{k})(d)=\operatorname{Prox}'_{r}(x^{k}+u^{k})d.
\end{align*}
Moreover, the equality $\operatorname{rge}\operatorname{Prox}'_{r}(x^{k}+u^{k})=\operatorname{rge}\overline{W}$ holds for large $k$. Combining Lemma~\ref{lemma-graphical-second-order-derivative}, this leads us to
\begin{align*}
\Gamma_{r}(\bar{x},\bar{u})(d)\leq\langle d,\overline{W}^{\dagger}d-d\rangle=\lim_{k\rightarrow\infty}\langle d,(\operatorname{Prox}'_{r}(x^{k}+u^{k}))^{\dagger}d-d\rangle=\lim_{k\rightarrow\infty}d^{2}r(x^{k},u^{k})(d),
\end{align*}
which ensures the fulfillment of the epigraph inclusion
\begin{align*}
\lim_{k\rightarrow\infty}\operatorname{epi}d^{2}r(x^{k},u^{k})\subset\operatorname{epi}\Gamma_{r}(\bar{x},\bar{u})
\end{align*}
and thus completes the proof of the lemma.
\end{proof}

Now we are in a position to establish the main second-order characterizations of strong variational sufficiency for local optimality in general composite optimization.

\begin{theorem}\label{thm:main-thm-SSOSC-strong-variational-sufficient-cond} Let $(\bar{x},\bar{u})$ solve the KKT system \eqref{eq:kkt-problem-P}. Suppose that $g$ satisfies Assumptions~{\rm\ref{assump-1} at $F(x)$ for all $u\in\partial g(F(x))$ with $F(x)$ in some neighborhood of $F(\bar{x})$ and Assumption~\ref{assump-2}} at $F(\bar{x})$ for $\bar{u}$. Then the following assertions are equivalent:

{\bf(i)} The strong variational sufficient condition for local optimality in $(P)$ holds at $\bar{x}$ for $\bar{u}$.

{\bf(ii)} The SSOSC \eqref{eq:SSOSC} holds at $\bar{x}$ for $\bar{u}$.

{\bf(iii)} There exists a sufficiently large number $\sigma>0$ such that all the matrices from the generalized Hessian $\Bar{\partial}^{2}_{xx}\mathcal{L}_{\sigma}(\bar{x},\bar{u})$ are positive-definite.
\end{theorem}
\begin{proof}
To verify (ii)$\Longrightarrow$(iii), deduce from \eqref{sovf1} and the imposed SSOSC at $\bar{x}$ for $\bar{u}$ that taking any $\bar{d}\ne 0$ with $F'(\bar{x})\bar{d}\in\operatorname{rge}\overline{W}=\operatorname{ker}\left(I-\overline{W}\overline{W}^{\dagger}\right)$ yields the existence of $\bar{\kappa}>0$ such that
\begin{align*}
\left\langle \nabla_{xx}^{2}L(\bar{x},\bar{u})\bar{d},\bar{d}\right\rangle+\Upsilon_{g}(F(\bar{x}),\bar{u})(F'(\bar{x})\bar{d})\geq\bar{\kappa}\|\bar{d}\|^{2}.
\end{align*}
It is trivial when $F'(\bar{x})\bar{d}=0$; therefore, we assume below that $F'(\bar{x})\bar{d}$ is nonzero. Consequently, Lemma~\ref{lem:postive-homogeneous-property} implies that for any $\bar{d}\in\mathcal{R}^{n}$ there exist $\kappa\in(0,\bar{\kappa}/2]$ and $c>0$ with
\begin{align}\label{eq:main-theorem-ssosc-ineq}
\left\langle \nabla_{xx}^{2}L(\bar{x},\bar{u})\bar{d},\bar{d}\right\rangle+\Upsilon_{g}(\bar{x},\bar{u})(F'(\bar{x})\bar{d})+c\|(I-\overline{W}\overline{W}^{\dagger})F'(\bar{x})\bar{d}\|^{2}\geq2\kappa\|\bar{d}\|^{2}.
\end{align}
Denoting $d:=F'(\bar{x})\bar{d}$, we deduce from 
Proposition~\ref{prop:relation-coderivative-element} that for any $v\in D^{*}\operatorname{Prox}_{\sigma g^{*}}(\bar{u}+\sigma F(\bar{x}))(d)$, there exists $U\in J\operatorname{Prox}_{g}(F(\bar{x})+\bar{u})$ satisfying 
\begin{align}\label{eq:equality-coderivative}
v-d=U(-d-(\sigma-1)v),\quad v=(I+(\sigma-1)U)^{-1}(I-U)d.
\end{align}
Furthermore, setting $y:=v-d$ and $z:=-d-(\sigma-1)v$ in 
Assumption~\ref{assump-2} gives us 
\begin{align*}
\langle \sigma v,d-v\rangle\geq\left\langle v-d,(\overline{W}^{\dagger}-I)(v-d)\right\rangle,
\end{align*}
which being combined with \eqref{eq:coderivative-norm-ineq} and \eqref{eq:equality-coderivative} yields
\begin{align*}
\sigma\langle v,d\rangle&\geq\left\langle v-d,(\overline{W}^{\dagger}-I)(v-d)\right\rangle+\sigma\langle v,v\rangle\\
&\geq\left\langle v-d,(\overline{W}^{\dagger}-I)(v-d)\right\rangle+\sigma\|(I-UU^{\dagger})d\|^{2}\\
&=\left\langle\sigma(I+(\sigma-1)U)^{-1}Ud,\sigma(\overline{W}^{\dagger}-I)(I+(\sigma-1)U)^{-1}Ud\right\rangle+\sigma\|(I-UU^{\dagger})d\|^{2}.
\end{align*}
Combining the latter with \eqref{sigma-cod-est} allows us to conclude that for any $\varepsilon>0$, there exists $\sigma>\bar{\sigma}$ with
\begin{align}\label{eq:main-theorem-1}
&\left\langle\sigma(I+(\sigma-1)U)^{-1}Ud,\sigma(\overline{W}^{\dagger}-I)(I+(\sigma-1)U)^{-1}Ud\right\rangle+\sigma\|(I-UU^{\dagger})d\|^{2}-\left\langle d,(\overline{W}^{\dagger}-I)d\right\rangle\nonumber\\
&\geq\left\langle UU^{\dagger}d,(\overline{W}^{\dagger}-I)UU^{\dagger}d\right\rangle-\varepsilon\|d\|^{2}+\sigma\|(I-UU^{\dagger})d\|^{2}-\left\langle d,(\overline{W}^{\dagger}-I)d\right\rangle.
\end{align}
Let $d:=d_{1}+d_{2}$, where $d_{1}\in\operatorname{rge}U$ and $d_{2}\in\operatorname{ker}(U)$. Then $d_{1}=UU^{\dagger}d$ and $d_{2}=(I-UU^{\dagger})d$. Consequently, we get the equalities
\begin{align}\label{eq:main-theorem-2}
\begin{array}{ll}
&\left\langle UU^{\dagger}d,(\overline{W}^{\dagger}-I)UU^{\dagger}d\right\rangle+\sigma\|(I-UU^{\dagger})d\|^{2}-\left\langle d,(\overline{W}^{\dagger}-I)d\right\rangle\\
&=\left\langle d_{1},(\overline{W}^{\dagger}-I)d_{1}\right\rangle+\sigma\|d_{2}\|^{2}-\left\langle d_{1}+d_{2},(\overline{W}^{\dagger}-I)(d_{1}+d_{2})\right\rangle\\
&=\sigma\|d_{2}\|^{2}-2\left\langle d_{1},(\overline{W}^{\dagger}-I)d_{2}\right\rangle-\left\langle d_{2},(\overline{W}^{\dagger}-I)d_{2}\right\rangle.
\end{array}
\end{align}
If $d_{2}=\overline{W}^{\dagger}d_{2}$, let $\kappa':=\kappa/\|F'(\bar{x})\|^{2}>0$ and then apply \eqref{eq:main-theorem-1} and \eqref{eq:main-theorem-2} to get the inequality
\begin{align}\label{eq:main-theorem-3}
&\left\langle\sigma(I+(\sigma-1)U)^{-1}Ud,\sigma(\overline{W}^{\dagger}-I)(I+(\sigma-1)U)^{-1}Ud\right\rangle+\sigma\|(I-UU^{\dagger})d\|^{2}+\kappa'\|d\|^{2}\nonumber\\
&\geq\left\langle d,(\overline{W}^{\dagger}-I)d\right\rangle+\eta\|d_{2}\|^{2}=\Upsilon_{g}(F(\bar{x}),\bar{u})(d)+\eta\|d_{2}\|^{2}
\end{align}
if $\sigma>0$ is sufficiently large and $\eta=\sigma$. In the case where $d_{2}\ne\overline{W}^{\dagger}d_{2}$, it follows that $d_{2}\ne 0$,
\begin{align}\label{eq:main-theroem-4}
\begin{array}{ll}
&\sigma\|d_{2}\|^{2}-2\left\langle d_{1},(\overline{W}^{\dagger}-I)d_{2}\right\rangle-\left\langle d_{2},(\overline{W}^{\dagger}-I)d_{2}\right\rangle\\
&\quad\qquad\geq\sigma\|d_{2}\|^{2}-2\|\overline{W}^{\dagger}-I\|\cdot\|d_{1}\|\cdot\|d_{2}\|-\|\overline{W}^{\dagger}-I\|\cdot\|d_{2}\|^{2},
\end{array}
\end{align}
and there exists $\sigma>\sigma'$ for which
\begin{align}\label{eq:main-theorem-5}
\left(\frac{1}{2}\sigma-\|\overline{W}^{\dagger}-I\|\right)\|d_{2}\|\ge 2\|\overline{W}^{\dagger}-I\|\cdot\|d_{1}\|.
\end{align}
Combining \eqref{eq:main-theorem-1}, \eqref{eq:main-theorem-2}, \eqref{eq:main-theroem-4}, and \eqref{eq:main-theorem-5} ensures that 
\eqref{eq:main-theorem-3} holds for $\sigma > \sigma’$ with $\eta = \frac{1}{2}\sigma$. The conducted analysis establishes the existence of a sufficiently large parameter $\sigma$ such that 
\begin{align}\label{eq:main-theorem-6}
\sigma\langle v,d\rangle+\kappa'\|d\|^{2}\geq\Upsilon_{g}(F(\bar{x}),\bar{u})(d)+\frac{\sigma}{2}\|d_{2}\|^{2}=\Upsilon_{g}(F(\bar{x}),\bar{u})(d)+\frac{\sigma}{2}\|(I-UU^{\dagger})d\|^{2}
\end{align}
for any $d\in\operatorname{rge}\overline{W}$ and $v\in D^{*}\operatorname{Prox}_{\sigma g^{*}}(\bar{u}+\sigma F(\bar{x}))(d)$. Consequently, it follows from \eqref{cod-est} of Theorem~\ref{prop-ineq-norm-W} along with \eqref{eq:main-theorem-6} and \eqref{eq:main-theorem-ssosc-ineq} that the inequality
\begin{align}\label{eq:main-theorem-7}
\left\langle \nabla_{xx}^{2}L(\bar{x},\bar{u})\bar{d},\bar{d}\right\rangle+\sigma \langle v ,d\rangle\geq2\kappa\|\bar{d}\|^{2}-\kappa'\|d\|^{2}\geq(2\kappa-\kappa'\|F'(\bar{x})\|^{2})\|\bar{d}\|^{2}=\kappa\|\bar{d}\|^{2}
\end{align}
holds uniformly for all $\sigma>0$ sufficiently large. Considering further any $V\in J\operatorname{Prox}_{\sigma g^{*}}(\bar{u}+\sigma F(\bar{x}))$ and using \eqref{cod-cl} ensure the existence of $v_{1},\ldots,v_{t}\in D^{*}\operatorname{Prox}_{\sigma g^{*}}(\bar{u}+\sigma F(\bar{x}))(d)$ such that $Vd=\sum\limits_{i=1}^{t}\lambda_{i}v_{i}$, where $\lambda_{i}\in[0,1]$ and $\sum\limits_{i=1}^{t}\lambda_{i}=1$. This yields $\langle d,Vd\rangle=\sum\limits_{i=1}^{t}\lambda_{i}\langle v_{i},d\rangle$ and hence leads us together with \eqref{eq:main-theorem-7} to the relationships
\begin{align*}
\left\langle \nabla_{xx}^{2}L(\bar{x},\bar{u})\bar{d},\bar{d}\right\rangle+\sigma\left\langle F'(\bar{x})\bar{d} ,VF'(\bar{x})\bar{d}\right\rangle&=\left\langle \nabla_{xx}^{2}L(\bar{x},\bar{u})\bar{d},\bar{d}\right\rangle+\sigma \left\langle d ,Vd\right\rangle\\
&=\left\langle \nabla_{xx}^{2}L(\bar{x},\bar{u})\bar{d},\bar{d}\right\rangle+\sigma\sum\limits_{i=1}^{t}\lambda_{i}\langle v_{i},d\rangle\geq\kappa\|\bar{d}\|^{2}.
\end{align*} 
The latter inequality readily demonstrates that, whenever $\sigma$ is sufficiently large, all the matrices from generalized Hessian $\Bar{\partial}^{2}_{xx}\mathcal{L}_{\sigma}(\bar{x},\bar{u})$ are positive-definite, which thus justifies the claimed implication (ii)$\Longrightarrow$(iii). The other two implications (iii)$\Longrightarrow$(i) and (i)$\Longrightarrow$(ii) follow directly from  Lemma~\ref{lem-quadratic-budle-Upsilon} in conjunction with \cite[Theorem~3]{Rockafellar2023} and \cite[Theorem~5]{Rockafellar2023}, respectively. This completes the proof of our main theorem.
\end{proof}

In the {\em nondegenerate} case where the matrix $F'(\ox)$ in (COP) is of full rank, the characterization of Theorem~\ref{thm:main-thm-SSOSC-strong-variational-sufficient-cond}(ii) reduces to the second-order characterization of strong variational sufficiency established in \cite[Theorem~6.2]{Mordukhovich2023}, which yields the SSOSC characterization of this property for NLPs under the classical LICQ. Note that Theorem~\ref{thm:main-thm-SSOSC-strong-variational-sufficient-cond} does not require either nondegeneracy like  LICQ, or any other qualification condition even in the case of NLPs. As shown in Appendix~\ref{appendix-A} below, the most involved Assumption~\ref{assump-2} holds in the settings where $g$ in problem (COP) is either the indicator function of the SDP cone, or the nuclear norm function. The constructive calculations provided in Examples~\ref{example-sdp} and \ref{example-nuclear}, respectively, lead us to the explicit second-order characterizations of strong variational sufficiency for the two major classes in nonpolyhedral constrained optimization {\em without any constraint qualifications}.

\section{Conclusions and Future Research}\label{sec:conc}

The paper discovers a general framework in composite optimization, where strong variational sufficiency for local optimality admits complete second-order characterizations via an appropriate SSOSC and the positive-definiteness of generalized Hessian matrices of the associated augmented Lagrangian. The imposed assumptions and obtained characterizations cover polyhedral and nonpolyhedral problems of composite optimizations without any nondegeneracy and/or other constraint qualifications and are explicitly specified for SDP and nuclear norm-based problems.

Our future research will be partly concentrate on revealing additional classes of nonpolyhedral problems of composite optimization that exhibit explicit second-order characterizations of strong variational sufficiency for local optimality without any nondegeneracy assumptions. We also intend to characterize the notion of (not strong) variational sufficiency for nonpolyhedral optimization problems.  Our main attention will be paid to developing second-order numerical algorithms with establishing their local convergence rates based on the obtained second-order characterizations of variational and strong variational sufficiency in composite optimization.

\begin{appendices}
\section{Calculations for SDP Cone and Nuclear Norm Functions}\label{appendix-A}

This appendix consists of two examples, which provide the verification of Assumption~\ref{assump-2} and explicit calculations of the parameters involving in the obtained second-order characterizations of local optimality in the two major classes in nonpolyhedral composite optimization. We begin with the case where $g$ in (COP) is the {\em indicator function of the SDP cone}.
    
\begin{example}\label{example-sdp} {\rm Let $S^{m}_{+}$ be the cone of positive-semidefinite matrices in the space of quadratic symmetric matrices of dimension $m$. Given $X\in S^{m}_{+}$, pick $U\in\mathcal{N}_{S^{m}_{+}}(X)$ and denote $A:=X+U$. In what follows, we verify the fulfillment of Assumption~\ref{assump-2} for (COP) with $g=\delta_{S^{m}_{+}}$ at $X$ for $U$. To proceed, let $\sigma_{i}$ $(i=1,\ldots,m)$ be the eigenvalues of $A$, and let $\Lambda$ be the diagonal matrix of the eigenvalues of $A$. Consider the spectral decomposition $A=P\Lambda P^{T}$, where $P$ is an orthogonal matrix of the corresponding orthonormal eigenvectors, and define the index sets
\begin{align*}
\alpha:=\{i\,|\,\sigma_{i}>0\},\quad\beta:=\{i\,|\,\sigma_{i}=0\},\quad\gamma:=\{i\,|\,\sigma_{i}<0\}.
\end{align*}
Pick $W\in J\Pi_{S^{m}_{+}}(X+U)$ (resp.\ $W\in\Bar\nabla\Pi_{S^{m}_{+}}(X+U)$ from \eqref{Bsub}) and $D\in S^{m}$ with $V=WD$ and then deduce from \cite[Proposition~2.2]{Sun2006} the existence of $Z\in J\Pi_{S^{|\beta|}_{+}}(0)$ (resp.\ $Z\in\Bar\nabla\Pi_{S^{|\beta|}_{+}}(0))$ such that
\begin{align*}
V=P\widetilde{V}P^{T},\quad \widetilde{V}=\left(\begin{array}{ccc}
\widetilde{D}_{\alpha\alpha}&\widetilde{D}_{\alpha\beta}&\Sigma_{\alpha\gamma}\circ\widetilde{D}_{\alpha\gamma}\\\widetilde{D}_{\beta\alpha}&Z(\widetilde{D}_{\beta\beta})&0\\\Sigma_{\gamma\alpha}\circ\widetilde{D}_{\gamma\alpha}&0&0
\end{array}\right),
\end{align*}
where '$\circ$' denotes the Hadamard product, and where
\begin{align*}
\widetilde{D}:=P^{T}DP,\quad\Sigma_{ij}:=\frac{\max\{\sigma_{i},0\}+\max\{\sigma_{j},0\}}{|\sigma_{i}|+|\sigma_{j}|},\ i,j=1,\ldots,m,
\end{align*}
with $0/0:=1$. Selecting $\alpha'=\beta$ in \cite[Lemma~11]{Pang2003} gives us $\overline{W}\in\Bar\nabla\Pi_{S^{m}_{+}}(X+U)$ with  $Z\in\Bar\nabla\Pi_{S^{|\beta|}_{+}}(0)$ such that $Z(\widetilde{D}_{\beta\beta})=\widetilde{D}_{\beta\beta}$. The first part of Assumption~\ref{assump-2} follows now from identity
\begin{align*}
&\left\{\left.P\left(\begin{array}{ccc}
\widetilde{D}_{\alpha\alpha}&\widetilde{D}_{\alpha\beta}&\Sigma_{\alpha\gamma}\circ\widetilde{D}_{\alpha\gamma}\\\widetilde{D}_{\beta\alpha}&\widetilde{D}_{\beta\beta}&0\\\Sigma_{\gamma\alpha}\circ\widetilde{D}_{\gamma\alpha}&0&0
\end{array}\right)P^{T}\right|\;\mbox{ for all }\;D\in S^{m},\,\widetilde{D}=P^{T}DP\right\}\\
&=\bigcup_{W\in J\Pi_{S^{m}_{+}}(X+U)}\operatorname{rge}W=\operatorname{rge}\overline{W}.
\end{align*}
Moreover, for any $Y\in\bigcup\limits_{W\in J\Pi_{S^{m}_{+}}(X+U)}\operatorname{rge}W$, there exist $W\in J\Pi_{S^{m}_{+}}(X+U)$ and $D\in S^{m}$ satisfying $Y=WD$ and $\widetilde{Y}=P^{T}YP$. Consequently, we obtain
\begin{align*}
\langle Y,D-Y\rangle=\langle\widetilde{Y},\widetilde{D}-\widetilde{Y}\rangle&=2\langle\widetilde{Y}_{\alpha\gamma},\widetilde{D}_{\alpha\gamma}-\widetilde{Y}_{\alpha\gamma}\rangle+\langle Z(\widetilde{D}_{\beta\beta}),\widetilde{D}_{\beta\beta}-Z(\widetilde{D}_{\beta\beta})\rangle\\
&\geq2\langle\widetilde{Y}_{\alpha\gamma},\widetilde{D}_{\alpha\gamma}-\widetilde{Y}_{\alpha\gamma}\rangle=-2\sum_{i\in\alpha,j\in\gamma}\frac{\sigma_{j}}{\sigma_{i}}\left(\widetilde{Y}_{ij}\right)^{2},
\end{align*}
where the inequality follows from \cite[Proposition~3.2]{Tangwang2024}. We can select the same matrix $\overline{W}\in\Bar\nabla\Pi_{S^{m}_{+}}(X+U)$ as above for which there is $\overline{D}\in S^{m}$ with $Y=\overline{W}\overline{D}$, $P^{T}\overline{D}P:=(d_{ij})_{m\times m}$ satisfying
\begin{align*}
d_{ij}=\widetilde{Y}_{ij},\;\forall\,i,j\in\alpha\cup\beta,\quad d_{ij}=d_{ji}=\frac{\sigma_{i}-\sigma_{j}}{\sigma_{i}}\widetilde{Y}_{ij},\;\forall\,i\in\alpha,\,j\in\gamma,
\end{align*}
\begin{align*}
\langle Y,\overline{D}-Y\rangle=\min_{{D\in S^{m},Y=WD}\atop{W\in J\Pi_{S^{m}_{+}}(X+U)}}\langle Y,D-Y\rangle.
\end{align*}
Thus the function $g=\delta_{S^{m}_{+}}$ satisfies Assumption~\ref{assump-2} at $X$ for $U$.}
\end{example}

The second example concerns the case of the {\em nuclear norm function} $g$ in (COP). 

\begin{example}\label{example-nuclear} {\rm 
Let $p,q\in\mathbb N$ with $p\leq q$, and let $X\in\mathcal{R}^{p\times q}$. In whet follows, we intend to show that Assumption~\ref{assump-2} holds for the nuclear norm function $g=\|\cdot\|_{*}$ at $X$ for $U\in\partial\|X\|_{*}$. Take $A=X+U$ and consider the singular value decomposition $A=R\,[\Sigma\ 0]\,S^{T}$, where $R\in\mathcal{R}^{p\times p}$ and $S\in\mathcal{R}^{q\times q}$ are orthogonal matrices, and where $\Sigma=\operatorname{Diag}(\sigma_{1},\ldots,\sigma_{p})$ is the diagonal matrix of singular values of $A$ with $\sigma_{1}\geq\cdots\ge\sigma_{p}\ge 0$.  
		
By \cite[Theorem~3]{Ma2011}, the proximal mapping of the nuclear norm $\|\cdot\|_{*}$ at $A$ is given by
\begin{align*}
\operatorname{Prox}_{\|\cdot\|_{*}}(A)=\overline{R}\,[\Sigma_{h}\ 0]\,\overline{S}^{T},
\end{align*}
where $\Sigma_{h}=\operatorname{Diag}(h(\sigma_{1}),\ldots,h(\sigma_{p}))$ and $h(t):=\max\{t-1,0\}$. Partition $\overline{S}$ into two parts as $\overline{S}=[\overline{S}_{1}\ \overline{S}_{2}]$ with $\overline{S}_{1}\in\mathcal{R}^{q\times p}$ and $\overline{S}_{2}\in\mathcal{R}^{q\times(q-p)}$. Define the index sets
\begin{align*}
    \alpha:=\{1,\ldots,p\},\quad\alpha_{1}:=\{i\in\alpha\,|\,\sigma_{i}>1\},\quad\alpha_{2}:=\{i\in\alpha\,|\,\sigma_{i}=1\},\quad\alpha_{3}:=\{i\in\alpha
,|\,\sigma_{i}<1\}.
\end{align*}
It follows from \cite[Proposition~2]{JiangSunToh2013} that for any $W\in J\operatorname{Prox}_{\|\cdot\|_{*}}(X+U)$ (resp.\ $W\in\Bar\nabla\operatorname{Prox}_{\|\cdot\|_{*}}(X+U))$  and $D\in\mathcal{R}^{p\times q}$ with $V=WD$, we get the representation
\begin{align}\label{eq:partial-proximal-nuclear-norm-times-D}
V=\overline{R}\left[\left(\Omega^{\alpha}\circ D_{1}^{s}+\Omega^{\gamma}\circ D_{1}^{a}\right)\overline{S}_{1}^{T}+\left(\Omega^{\beta}\circ D_{2}\right)\overline{S}_{2}^{T}\right],
\end{align}
where $D_{1}^{s}:=\frac{D_{1}+D_{1}^{T}}{2}$, $D_{1}^{a}:=\frac{D_{1}-D_{1}^{T}}{2}$, $D_{1}:=\overline{R}^{T}D\overline{S}_{1}$, $D_{2}:=\overline{R}^{T}D\overline{S}_{2}$ with $\overline{R}^{T}D\overline{S}=(d_{ij})_{p\times q}$, and where $\Omega^{\alpha}\in\mathcal{R}^{p\times p}$, $\Omega^{\gamma}\in\mathcal{R}^{p\times p}$, $\Omega^{\beta}\in\mathcal{R}^{p\times(q-p)}$ are defined as
\begin{align*}
&\Omega^{\alpha}:=\left[\begin{array}{ccc}\Omega^{\alpha}_{\alpha_{1}\alpha_{1}} & \Omega^{\alpha}_{\alpha_{1}\alpha_{2}} & \Omega^{\alpha}_{\alpha_{1}\alpha_{3}}\\
(\Omega^{\alpha}_{\alpha_{1}\alpha_{2}})^{T} & \Omega^{\alpha}_{\alpha_{2}\alpha_{2}} & 0\\
(\Omega^{\alpha}_{\alpha_{1}\alpha_{3}})^{T} & 0 & 0\end{array}\right]\;\mbox{ with }\;
\Omega^{\alpha}_{ij}:=1\;\mbox{ for all }\;i\in\alpha_{1},\,j\in\alpha_{1}\cup\alpha_{2},\\
&\Omega_{ij}^{\alpha}:=\frac{\sigma_{i}-1}{\sigma_{i}-\sigma_{j}}\;\mbox{ for all }\;i\in\alpha_{1},\,j\in\alpha_{3},\quad\Omega_{ij}^{\alpha}:=\Omega^{\alpha}_{ji}\in[0,1]\ (\mbox{resp.}\ \{0,1\})\;\mbox{ for all }\;i,j\in\alpha_{2},\\
&\Omega^{\gamma}:=\left[\begin{array}{ccc}
\Omega^{\gamma}_{\alpha_{1}\alpha_{1}} & \Omega^{\gamma}_{\alpha_{1}\alpha_{2}} & \Omega^{\gamma}_{\alpha_{1}\alpha_{3}}\\ (\Omega^{\gamma}_{\alpha_{1}\alpha_{2}})^{T} & 0 & 0 \\ (\Omega^{\gamma}_{\alpha_{1}\alpha_{3}})^{T} & 0 & 0
\end{array}\right],\quad
\Omega_{ij}^{\gamma}:=\frac{h(\sigma_{i})+h(\sigma_{j})}{\sigma_{i}+\sigma_{j}}\;\mbox{ for all }\;i\in\alpha_{1},\,j\in\alpha,\\
&\Omega^{\beta}:=\left[\begin{array}{c}
\Omega^{\beta}_{\alpha_{1}\bar{\beta}}\\ 0
\end{array}\right],\quad\bar{\beta}:=\{1,\ldots,q-p\},\quad
\Omega_{ij}^{\beta}:=\frac{\sigma_{i}-1}{\sigma_{i}}\;\mbox{ for all }\;i\in\alpha_{1},\,j\in\bar{\beta}.		
\end{align*}
Since $A=R[\Sigma\ 0]S^{T}=\overline{R}[\Sigma\ 0]\overline{S}^{T}$, it follows from
\cite[Proposition~2.4]{ChenLiuSunToh2016} that there exist orthogonal matrices $Q_{1}\in\mathcal{R}^{|\alpha_{1}|\times |\alpha_{1}|}$, $Q_{2}\in\mathcal{R}^{|\alpha_{2}|\times|\alpha_{2}|}$, $Q_{3}\in\mathcal{R}^{|\alpha_{3}|\times |\alpha_{3}|}$, $Q_{4}\in\mathcal{R}^{(q-2p)\times(q-2p)}$, and $Q_{5}\in\mathcal{R}^{(q-2p)\times(q-2p)}$ for which we have the expressions
\begin{align*}
\overline{R}&=R\left[\begin{array}{cccc}
Q_{1} & 0 & 0 & 0\\
0 & Q_{2} & 0 & 0\\
0 & 0 & Q_{3} & 0\\
0 & 0 & 0 & Q_{4}
\end{array}\right],\quad
\overline{S}=S\left[\begin{array}{cccc}
Q_{1} & 0 & 0 & 0\\
0 & Q_{2} & 0 & 0\\
0 & 0 & Q_{3} & 0\\
0 & 0 & 0 & Q_{5}
\end{array}\right].
\end{align*}
Therefore, this brings us to the representation
\begin{align*}
&\overline{R}[Z_{1}(s)\ Z_{2}(s)]\overline{S}^{T}\\
&=R\left[\begin{array}{cccc}
Q_{1} & 0 & 0 & 0\\
0 & Q_{2} & 0 & 0\\
0 & 0 & Q_{3} & 0\\
0 & 0 & 0 & Q_{4}
\end{array}\right][\Omega^{\alpha}\circ D^{s}_{1}+\Omega^{\gamma}\circ D^{a}_{1}\ \ \Omega^{\beta}\circ D_{2}]\left[\begin{array}{cccc}
Q_{1}^{T} & 0 & 0 & 0\\
0 & Q_{2}^{T} & 0 & 0\\
0 & 0 & Q_{3}^{T} & 0\\
0 & 0 & 0 & Q_{5}^{T}
\end{array}\right]S^{T}.
\end{align*}
For any nonsingular matrices $P_{1}\in\mathcal{R}^{m\times m}$, $P_{2}\in\mathcal{R}^{n\times n}$, and $B,\widetilde{B},C\in\mathcal{R}^{m\times n}$ with all nonzero elements of $B$ and $B\circ\widetilde{B}=1_{m\times n}$, we select the matrix
\begin{align*}
\widetilde{C}:=\widetilde{B}\circ(P_{1}^{-1}(B\circ P_{1}CP_{2}))P_{2}^{-1}
\end{align*}
such that $B\circ P_{1}CP_{2}=P_{1}(B\circ \widetilde{C})P_{2}$. Hence there exists $\widetilde{D}$ with
\begin{align}\label{fixed-R-S}
\overline{R}[\Omega^{\alpha}\circ D^{s}_{1}+\Omega^{\gamma}\circ D^{a}_{1}\quad \Omega^{\beta}\circ D_{2}]\overline{S}^{T}=R[\Omega^{\alpha}\circ \widetilde{D}^{s}_{1}+\Omega^{\gamma}\circ \widetilde{D}^{a}_{1}\quad \Omega^{\beta}\circ \widetilde{D}_{2}]S^{T},
\end{align}
where $\widetilde{D}_{1}^{s}:=\frac{\widetilde{D}_{1}+\widetilde{D}_{1}^{T}}{2}$, $\widetilde{D}_{1}^{a}:=\frac{\widetilde{D}_{1}-\widetilde{D}_{1}^{T}}{2}$, $\widetilde{D}_{1}:=R^{T}\widetilde{D}S_{1}$, and $\widetilde{D}_{2}:=R^{T}\widetilde{D}S_{2}$. Let $\overline{R}:=R$, $\overline{S}:=S$ and then find $\widetilde{D}$ satisfying the equality \eqref{fixed-R-S}.
Thus there exists $\overline{W}\in\Bar\nabla\operatorname{Prox}_{\|\cdot\|_{*}}(X+U)$ with $\Omega^{\alpha}=\widehat{\Omega}^{\alpha}$, $\widehat{\Omega}_{ij}^{\alpha}=1$ as $i,j\in\alpha_{2}$, and 
\begin{align*}
\bigcup_{W\in J\operatorname{Prox}_{\|\cdot\|_{*}}(X+U)}\operatorname{rge}W&=\operatorname{rge}\overline{W}\\
&=\left\{\left.R\left[\left(\widehat{\Omega}^{\alpha}\circ D_{1}^{s}+\Omega^{\gamma}\circ D_{1}^{a}\right)S_{1}^{T}+\left(\Omega^{\beta}\circ D_{2}\right)S_{2}^{T}\right]\,\right|\,D\in\mathcal{R}^{p\times q}\right\},
\end{align*}
which ready justifies the first part of Assumption~\ref{assump-2}.

Furthermore, for any $Y\in\bigcup\limits_{W\in J\operatorname{Prox}_{\|\cdot\|_{*}}(X+U)}\operatorname{rge}W$, there exist $W\in J\operatorname{Prox}_{\|\cdot\|_{*}}(X+U)$ and $D\in\mathcal{R}^{p\times q}$ such that $Y=WD$. Then it follows from \eqref{eq:partial-proximal-nuclear-norm-times-D} that
\begin{align}\label{eq:minmal-inner-product-Y-D-Y}
\langle Y,D-Y\rangle&=\left\langle \left(\Omega^{\alpha}\circ D_{1}^{s}+\Omega^{\gamma}\circ D_{1}^{a}\right)S_{1}^{T}+\left(\Omega^{\beta}\circ D_{2}\right)S_{2}^{T},\right.\nonumber\\
&\quad\left.R^{T}D-(\left(\Omega^{\alpha}\circ D_{1}^{s}+\Omega^{\gamma}\circ D_{1}^{a}\right)S_{1}^{T}+\left(\Omega^{\beta}\circ D_{2}\right)S_{2}^{T})\right\rangle\nonumber\\
&=\left\langle[\Omega^{\alpha}\circ D_{1}^{s}+\Omega^{\gamma}\circ D_{1}^{a}\quad \Omega^{\beta}\circ D_{2}],R^{T}DS-[\Omega^{\alpha}\circ D_{1}^{s}+\Omega^{\gamma}\circ D_{1}^{a}\quad \Omega^{\beta}\circ D_{2}]\right\rangle\nonumber\\
&=\left\langle \Omega^{\alpha}\circ D_{1}^{s}+\Omega^{\gamma}\circ D_{1}^{a},D_{1}-(\Omega^{\alpha}\circ D_{1}^{s}+\Omega^{\gamma}\circ D_{1}^{a})\right\rangle + \left\langle \Omega^{\beta}\circ D_{2}, D_{2}-\Omega^{\beta}\circ D_{2}\right\rangle\nonumber\\
&=2\sum_{{i<j,i\in\alpha_{1}}\atop{j\in\alpha_{1}\cup\alpha_{2}}}\left(1-\frac{h(\sigma_{i})+h(\sigma_{j})}{\sigma_{i}+\sigma_{j}}\right)\frac{h(\sigma_{i})+h(\sigma_{j})}{\sigma_{i}+\sigma_{j}}\left(\frac{d_{ij}-d_{ji}}{2}\right)^{2}\nonumber\\
&\quad+2\sum_{{i\in\alpha_{1}}\atop{j\in\alpha_{3}}}\left[\frac{(\sigma_{i}-1)(1-\sigma_{j})}{(\sigma_{i}-\sigma_{j})^{2}}\left(\frac{d_{ij}+d_{ji}}{2}\right)^{2}+\frac{(\sigma_{i}-1)(1+\sigma_{j})}{(\sigma_{i}+\sigma_{j})^{2}}\left(\frac{d_{ij}-d_{ji}}{2}\right)^{2}\right]\nonumber\\
&\quad+\sum_{i\in\alpha_{2}}\Omega^{\alpha}_{ii}(1-\Omega^{\alpha}_{ii})d_{ii}^{2}+2\sum_{{i<j}\atop{i\in\alpha_{2},j\in\alpha_{2}}}\Omega^{\alpha}_{ij}(1-\Omega^{\alpha}_{ij})d_{ij}^{2}+\sum_{i\in\alpha_{1}}\frac{\sigma_{i}-1}{\sigma_{i}^{2}}\sum_{j=p+1}^{q}d_{ij}^{2}\nonumber\\
&\geq2\sum_{{i<j,i\in\alpha_{1}}\atop{j\in\alpha_{1}\cup\alpha_{2}}}\left(1-\frac{h(\sigma_{i})+h(\sigma_{j})}{\sigma_{i}+\sigma_{j}}\right)\frac{h(\sigma_{i})+h(\sigma_{j})}{\sigma_{i}+\sigma_{j}}\left(\frac{d_{ij}-d_{ji}}{2}\right)^{2}+\sum_{i\in\alpha_{1}}\frac{\sigma_{i}-1}{\sigma_{i}^{2}}\sum_{j=p+1}^{q}d_{ij}^{2}\nonumber\\
&\quad+2\sum_{{i\in\alpha_{1}}\atop{j\in\alpha_{3}}}\left[\frac{(\sigma_{i}-1)(1-\sigma_{j})}{(\sigma_{i}-\sigma_{j})^{2}}\left(\frac{d_{ij}+d_{ji}}{2}\right)^{2}+\frac{(\sigma_{i}-1)(1+\sigma_{j})}{(\sigma_{i}+\sigma_{j})^{2}}\left(\frac{d_{ij}-d_{ji}}{2}\right)^{2}\right]\nonumber
\end{align}
\begin{align}
&=\left\langle \widehat{\Omega}^{\alpha}\circ D_{1}^{s}+\Omega^{\gamma}\circ D_{1}^{a},D_{1}-(\widehat{\Omega}^{\alpha}\circ D_{1}^{s}+\Omega^{\gamma}\circ D_{1}^{a})\right\rangle + \left\langle \Omega^{\beta}\circ D_{2}, D_{2}-\Omega^{\beta}\circ D_{2}\right\rangle\nonumber\\
&=2\sum_{{i<j,i\in\alpha_{1}}\atop{j\in\alpha_{1}\cup\alpha_{2}}}\left(\frac{\sigma_{i}+\sigma_{j}}{h(\sigma_{i})+h(\sigma_{j})}-1\right)\left(\frac{\widetilde{Y}_{ij}-\widetilde{Y}_{ji}}{2}\right)^{2}\\	&\quad+2\sum_{{i\in\alpha_{1}}\atop{j\in\alpha_{3}}}\left[\frac{1-\sigma_{j}}{\sigma_{i}-1}\left(\frac{\widetilde{Y}_{ij}+\widetilde{Y}_{ji}}{2}\right)^{2}+\frac{1+\sigma_{j}}{\sigma_{i}-1}\left(\frac{\widetilde{Y}_{ij}-\widetilde{Y}_{ji}}{2}\right)^{2}\right]+\sum_{i\in\alpha_{1}}\frac{1}{\sigma_{i}-1}\sum_{j=p+1}^{q}\widetilde{Y}_{ij}^{2},\nonumber
\end{align}
where $\widetilde{Y}:=R^{T}YS$, and where $\widehat{\Omega}^{\alpha}$ satisfies $\widehat{\Omega}^{\alpha}_{ij}=1$ for all $i,j\in\alpha_{2}$. Using \eqref{eq:minmal-inner-product-Y-D-Y} allows us to choose the same matrix $\overline{W}\in\Bar\nabla
\operatorname{Prox}_{\|\cdot\|_{*}}(X+U)$ as above such that there exists $\overline{D}\in\mathcal{R}^{q\times q}$ with $R^{T}\overline{D}S=(d_{ij})_{p\times q}$ for which we have
\begin{align*}
d_{ij}&=\frac{\widetilde{Y}_{ij}+\widetilde{Y}_{ji}}
{2}+\frac{\sigma_{i}+\sigma_{j}}{h(\sigma_{i})+h(\sigma_{j})}\frac{\widetilde{Y}_{ij}-\widetilde{Y}_{ji}}{2}\;\mbox{ if }\;i\in\alpha_{1},\,j\in\alpha_{1}\cup\alpha_{2},\\
d_{ji}&=\frac{\widetilde{Y}_{ij}+\widetilde{Y}_{ji}}{2}-\frac{\sigma_{i}+\sigma_{j}}{h(\sigma_{i})+h(\sigma_{j})}\frac{\widetilde{Y}_{ij}-\widetilde{Y}_{ji}}{2}\;\mbox{ if }\;i\in\alpha_{1},\,j\in\alpha_{1}\cup\alpha_{2},\\
\frac{d_{ij}+d_{ji}}{2}&=\widetilde{Y}_{ij}\;\mbox{ if }\;i,j\in\alpha_{2},\\
d_{ij}&=\frac{\sigma_{i}-\sigma_{j}}{\sigma_{i}-1}\frac{\widetilde{Y}_{ij}+\widetilde{Y}_{ji}}
{2}+\frac{\sigma_{i}+\sigma_{j}}{\sigma_{i}-1}\frac{\widetilde{Y}_{ij}-\widetilde{Y}_{ji}}{2}\;\mbox{ if }\;i\in\alpha_{1},\,j\in\alpha_{3},\\
d_{ji}&=\frac{\sigma_{i}-\sigma_{j}}{\sigma_{i}-1}\frac{\widetilde{Y}_{ij}+\widetilde{Y}_{ji}}{2}-\frac{\sigma_{i}+\sigma_{j}}{\sigma_{i}-1}\frac{\widetilde{Y}_{ij}-\widetilde{Y}_{ji}}{2}\;\mbox{ if }\;i\in\alpha_{1},\,j\in\alpha_{3},\\	d_{ij}&=\frac{\sigma_{i}}{\sigma_{i}-1}\widetilde{Y}_{ij}\;\mbox{ if }\;i\in\alpha_{1},\,j\in\{p+1,\ldots,q\},
\end{align*}
where $\overline{W}\overline{D}=Y$ and the following equalities holds:
\begin{align*}
\min_{{D\in S^{m},Y=WD}\atop{W\in J\operatorname{Prox}_{\|\cdot\|_{*}}(X+U)}}\langle Y,D-Y\rangle&=\langle Y,\overline{D}-Y\rangle=2\sum_{{i<j,i\in\alpha_{1}}\atop{j\in\alpha_{1}\cup\alpha_{2}}}\left(\frac{\sigma_{i}+\sigma_{j}}{h(\sigma_{i})+h(\sigma_{j})}-1\right)\left(\frac{\widetilde{Y}_{ij}-\widetilde{Y}_{ji}}{2}\right)^{2}\\	&+2\sum_{{i\in\alpha_{1}}\atop{j\in\alpha_{3}}}\left[\frac{1-\sigma_{j}}{\sigma_{i}-1}\left(\frac{\widetilde{Y}_{ij}+\widetilde{Y}_{ji}}{2}\right)^{2}+\frac{1+\sigma_{j}}{\sigma_{i}-1}\left(\frac{\widetilde{Y}_{ij}-\widetilde{Y}_{ji}}{2}\right)^{2}\right]\\
&\quad+\sum_{i\in\alpha_{1}}\frac{1}{\sigma_{i}-1}\sum_{j=p+1}^{q}\widetilde{Y}_{ij}^{2}.
\end{align*}
This verifies the second part of Assumption \ref{assump-2} for $g=\|\cdot\|_{*}$.}
\end{example}
\end{appendices}

\bibliographystyle{amsplain}
\bibliography{ref.bib}

\end{document}